\documentclass[aap,MSNbibl,dvips]{arximspdf}
\usepackage{graphicx}
\doi{10.1214/13-AAP946} %kopijuoti is PTS
\volume{24}
\issue{3}
\pubyear{2014}
\firstpage{1226}
\lastpage{1268}

\makeatletter

\newcommand{\calI}{\mathcal{I}}
\newcommand{\bi}{\operatorname{bia}}
\newcommand{\ba}{\operatorname{bal}}
\newcommand{\bE}{\mathbf{E}}
\newcommand{\bP}{\mathbf{P}}

\newcommand{\ve}{\varepsilon}
\newcommand{\al}{\alpha}
\newcommand{\ep}{\varepsilon}
\newcommand{\ka}{\kappa}
\newcommand{\la}{\lambda}
\newcommand{\vp}{\varphi}
\newcommand{\om}{\omega}
\newcommand{\tila}{\tilde{a}}
\newcommand{\tsig}{\tilde{\sigma}}
\newcommand{\ttau}{\tilde{\tau}}
\newcommand{\tka}{\tilde{\ka}}
\newcommand{\tgamma}{\tilde{\gamma}}
\newcommand{\tphi}{\tilde{\phi}}

\newcommand{\tX}{\tilde{X}}
\newcommand{\hY}{\hat{Y}}
\newcommand{\tT}{\tilde{T}}
\newcommand{\tr}{\tilde{r}}
\newcommand{\eqref}[1]{(\ref{#1})}

\newproclaim{example}{Example}

\newtheorem{lemma}{Lemma}[section]
\newtheorem{proposition}{Proposition}[section]
\newtheorem{theorem}{Theorem}[section]

\newproclaim{definition}{Definition}
\makeatother

\begin{document}
\begin{frontmatter}

\title{Stochastically-induced bistability in chemical reaction
systems\thanksref{T1}}
\runtitle{Stochastic bistability in chemical reaction systems}
\thankstext{T1}{Supported by the New Researchers Grant (F00637-126232)
from FQRNT (Fonds de recherche du Qu\'ebec Nature et technologies).}

\begin{aug}
\author[A]{\fnms{John K.} \snm{McSweeney}\ead[label=e2]{mcsweene@rose-hulman.edu}}
\and
\author[B]{\fnms{Lea} \snm{Popovic}\corref{}\ead[label=e1]{lpopovic@mathstat.concordia.ca}}
\runauthor{J. K. McSweeney and L. Popovic}
\affiliation{Rose-Hulman Institute of Technology and Concordia University}
\address[A]{Department of Mathematics\\
Rose-Hulman Institute of Technology\\
5500 Wabash Ave.\\
Terre Haute, Indiana 47803\\
USA\\
\printead{e2}}
%adresu isvedimo komanda gale!
\address[B]{Department of Mathematics and Statistics\\
Concordia University\\
1455 de Maisonneuve Blvd. West\\
Montreal, QC H3G 1M8\\
Canada\\
\printead{e1}}
\end{aug}

% HISTORY:
\received{\smonth{5} \syear{2012}}
\revised{\smonth{6} \syear{2013}}

% ABSTRACT
%
\begin{abstract}
We study a stochastic two-species chemical reaction system with two
mechanisms. One mechanism consists of chemical interactions which
govern the overall drift of species amounts in the system; the other
mechanism consists of resampling, branching or splitting which makes
unbiased perturbative changes to species amounts. Our results show that
in a system with a large but bounded capacity, certain combinations of
these two types of interactions can lead to stochastically-induced
bistability. Depending on the relative magnitudes of the rates of these
two sets of interactions, bistability can occur in two distinct ways
with different dynamical signatures.
\end{abstract}
%
% KEYWORDS

\end{frontmatter}

%s1 #&#
\section{Introduction}

%Dynamics leading to ON/OFF values and bistability:
Recent advances in measurement technology have enabled scientists to
observe molecular dynamics in single cells and to study the
cell-to-cell variability (Brehm-Stecher and Johnson \cite{B-S}).
Many studies have shown that variability observed in genetically
identical cells is due to noise that is inherent to biochemical
reactions happening within each cell (McAdams and Arkin \cite{MA},
Elowitz et al. \cite{El}).
Understanding how intracellular mechanisms are affected by this
intrinsic noise is an important challenge for systems biology.
Determining what role this noise plays in creating phenotypic
heterogeneity has many practical consequences (Avery \cite{Av}).

An important feature in cellular dynamics is bistability, the
alternation between two different stable states for a molecular species.
This feature is present in many gene-expression systems, where a gene
alternates between two types of states (``on'' and ``off'') regulating
the production of a protein. It is also present in many phosphorylation
switches in signaling pathways.
Causes for bistable behavior can be deterministic, but many bistable
switching patterns are enabled by stochastic fluctuations.
It is often assumed that it follows from the existence of two stable
equilibria in the deterministic drift and the ability of infrequent
large fluctuations to pull the system from a basin of attraction of one
equilibirum to the other.
There are also cases of chemical dynamics in which bistability is not
possible in the deterministic model, but is possible in the stochastic
model of the same chemical reaction system (Samoilov et al. \cite{SPA},
Bishop and Qian \cite{BQ}). Metastable behavior is also sometimes
observed (Robert et al. \cite{LydR}).

In addition to noise inherent to biochemical reactions, cells also
experience fluctuations in molecular composition due to cell division.
This source of noise is significant, and also difficult to separate
from the noise due to biochemical reactions (Huh and Paulsson \cite
{Pa1,Pa2}).
In this paper we investigate under what conditions a system of chemical
reactions in a cell can use these two sources of noise to exhibit
bistable or metastable behavior in their molecular composition.

We would like to emphasize a couple of points observed in the
literature. First, the rate of switching between two states is
important for cellular development and survival (Acar et al. \cite{vO}).
Time-scales on which transitions between stable states happen varies
whole orders of magnitude over different systems. For example, in the
lysogenic state of \textit{E. coli} the time-scale of switching between
states is slow (Zong et al. \cite{ZSSSG})---once per $10^8$ cell
generations---as determined from the activity of a controlling protein.
%%switches once per $10^8$ cell generations, switching time is
%exponential in the cell doubling time, weak-noise conditions,
%barrier-crossing, noise source from mRNA bursting
In the case of gene expression in \textit{S. cerevisiae} the switching
time-scale is fast (Kaufmann et al. \cite{KYMvO})---once per $8.33$
generations---and switching times between mother and daughter cells are
correlated in a way that takes several generations to dissipate. %0.12
%switches per generation=1 switch per 8.33 generations, switching rate
%is linear in ?, long time-scale of correlations which increase as the
%noise increases, large fluctuations in the regulatory proetin, noise
%source from fluctuation of regulatory proteins. Varying orders of
%magnitude have been observed for switching time-scales in different
%systems.
Second, both the strength and the distribution of noise affects whether
bistability will occur and what the final outcomes will be. Samoilov et
al. \cite{SPA} and Bishop and Qian \cite{BQ} show that auxiliary
chemical reactions can induce a dynamic switching behavior in the
enzymatic PdP cycle, and that final dynamics is determined by the noise
of the additional reactions. In the bistable switch of lactose operon
of \textit{E. coli} Robert et al. \cite{LydR} show that both cellular
growth rate and the molecular concentration levels influence the
ability to switch. Huh and Paulsson \cite{Pa2} showed that the type
of the cellular division mechanism also plays an important role in the
form of the final dynamics. We interpret these observations vis-a-vis
our results in the Discussion section.

Finally, we note that bistablility in a stochastic population system is
not limited to chemical dynamics. In a genetic population, mutation and
selection may lead to alternating fixation in one of two genotypes. In
an ecological population, interactions between species can lead to
dynamics where two competing species are switching for dominance. We
note that our analysis and results apply to any population model
described by a density dependent Markov jump process.

%s1.1 #&#
\subsection{Outline of results}

We examine qualitatively different ways in which switching between
stable states is a result of a stochastic effect in a population
modeled by density dependent Markov jump processes. In addition to
noise inherent to the reaction system, we include an intrinsically
noisy splitting/resampling mechanism in the system. In many stochastic
branching models an entity will (upon reproduction, division,
duplication, etc.) produce offspring identical to itself. Here we model
the division as unbiased but variable. When a cell divides its
molecules are randomly allocated to its daughter cells, only on average
replicating the parent's molecular composition. We will show that
introducing such a splitting process at a sufficiently high rate can
produce switching dynamics in which previously unattainable states
become attainable. We will exploit the fact that these two sets of
mechanisms (reactions in the system and changes due to unbiased
resampling/splitting of the system) may operate on different
time-scales.\looseness=1

We consider the following question: which qualitatively different types
of behavior can we observe and under which time scaling regimes? The
short answer is as follows: (1) If the resampling mechanism is
``slower'' than the reaction dynamics, then the system behavior will
entirely depend on the nonlinear dynamics of the reactions: in case the
underlying deterministic system has multiple stable equilibria, the
stochastic process will behave as a Markov chain switching between
these states. (2) If the resampling mechanism is much ``faster'' than
the reaction dynamics, then the system behavior will not depend on the
details of the reaction dynamics, and will behave as a Markov chain
switching between two extremes (zero and capacity) of the system. We
define a single parameter based on the rates of the two mechanisms that
makes the meaning of ``faster'' and ``slower'' in the statements above
mathematically precise.

We show that a fast but unbiased resampling mechanism may be necessary
to produce bistable behavior that the reaction dynamics cannot exhibit.
We further show that the two cases, (1) and (2), produce qualitatively
different dynamical signatures, in terms of switching times and stable
points. Since our analysis only depends on general features
(unbiasedness and time-scale of the rate) of the resampling mechanism,
one can also use a set of auxiliary reactions instead of resampling.
There are other types of noisy mechanisms that one could consider;
however, our goal is to stress that adding noise with even small
changes (relative to the size of the system) can produce bistable
behavior. The additional noise achieves this either by: (1) introducing
small perturbations to a dynamical system that already has the required
properties for bistability or (2) occurring so frequently that the
details of the dynamical system are irrelevant and the system is pushed
to its extreme (zero or capacity) amounts.

%s2 #&#
\section{Description of the process}

%s2.1 #&#
\subsection{Stochastic model for reaction dynamics}\label{subsec:reacmodel}
In the customary notation for interaction of chemical species labeled
$A, B, \ldots,$ %the expression
%e1 #&#
\begin{equation}
\label{basicreac} \bigl\{a_{i}A+b_{i}B+\cdots \longrightarrow
a_{i}'A+b_{i}'B+\cdots\bigr\}
_{i=1,\ldots, k}
\end{equation}
denotes a system of reactions indexed by $i=1,\ldots, k$ in which
$a_{i},b_{i},\ldots\in\mathbb{Z}^+$ molecules of types $A,B,\ldots$
respectively react and produce $a_{i}', b_{i}',\ldots\in\mathbb{Z}^+$
molecules of these types.
Each reaction $i$ has a reaction rate $\la_i$, a time and
state-dependent rate of occurrences of this reaction. If
$X_A(t),X_B(t),\ldots$ denote the number of molecules of type
$A,B,\ldots
$ respectively at time $t\geq0$, then $X(t)=(X_A(t),X_B(t),\ldots)$
evolves as a Markov Jump Process with jump sizes $\{
(a_{i}'-a_{i},b_{i}'-b_{i},\ldots)\}$ occurring at rates $\la
_i(X_A(t),X_B(t),\ldots)$.

The reaction formalism \eqref{basicreac} can also be used to describe
other systems of interacting entities under a well-mixed spatial
assumption. For example, evolution of an SIS epidemic is expressed as
$I+S \to2I$ (infection), $I \to S$ (recovery); a~two-allele Moran
model with mutation from population genetics can be expressed as $A\to
B, B\to A$ (mutation), $A+B\to2A, A+B\to2B$ (resampling).

For simplicity, we consider the effect of a system of chemical
reactions on essentially a single molecular species $A$. We include the
effect of only one other species $B$ which satisfies a conservation
relation with $A$. This means that every reaction involving $A$ and $B$
is of the form $aA + bB \to(a+\zeta)A + (b-\zeta)B$ for some $\zeta
\in\{-a,\ldots,b\}$, and it ensures that the state space of our system
is one-dimensional determined by $(X_A(t),t\ge0)$. The rationale for
such a conservation law could come from a cellular environment which is
limited (by a factor such as space, or availability of nutrients or
catalysts), or a molecular species whose type can take two different
forms (e.g., a gene that has two allelic types).

We also assume the following properties for the reaction dynamics:
\begin{longlist}[(4)]
\item[(1)] The amount of species $X_A(t)$ is bounded above by the
system capacity $N$ and below by $0$. The rate of any reaction that
decreases the amount of $A$ is zero when $X_A=0$, and the rate of any
reaction that increases the amount of $A$ is zero when $X_A=N$.
\item[(2)] The drift at $0$ and $N$ of the overall reaction dynamics is
directed toward the interior
\[
\frac{d}{ds}\bE\bigl[X_A(s)|X_A(t)=0\bigr]
|_{s=t}> 0,\qquad \frac{d}{ds}\bE \bigl[X_A(s)|X_A(t)=N
\bigr] |_{s=t}< 0.
\]
\item[(3)] The form of reaction rates $\la$ is governed by the law of
(stochastic) mass-action kinetics. A reaction of the form
\[
aA+bB \stackrel{\ka} {\to} a'A+b'B
\]
has rate $\la(X(t)) = \ka(X_A(t))_a (X_B(t))_b = \ka(X_A(t))_a
(N-X_A(t))_b$. Here $(Z)_c$ denotes the falling factorial $(Z)_c =
Z(Z-1)\cdots(Z-c+1)$. When we renormalize $X_A(t)$ by its maximum value
$N$, we will also need the ``scaled falling factorial'' $(z)_{c,N}$
defined by
%
%e2 #&#
\begin{eqnarray}
\label{sff} (z)_{c,N}:= N^{-c}(Nz)_c = z
\biggl(z-\frac{1}N \biggr) \biggl(z-\frac{2}N \biggr)\cdots
\biggl(z-\frac{c-1}{N} \biggr),
\nonumber
\\[-8pt]
\\[-8pt]
\eqntext{0\leq z \leq1.}
\end{eqnarray}
Note that $\lim_{N \to\infty}(z)_{c,N} = z^c$ for fixed $z$ and $c$.
The constant $\ka>0$ is independent of the state
$(X_A(t),X_B(t)=N-X_A(t))$ but will depend on the scaling parameter~$N$,
$\ka=\ka(N)$. We do not necessarily assume that $\ka(N)$ has the
``standard'' scaling form $\ka(N)=\tka N^{1-(a+b)}$.
\item[(4)] The effect on $A$ from any other species in the system is
subsumed into the values of the rate constants $\ka$, and are assumed
to be state-independent.
\end{longlist}

Assumption~(1) ensures that $X_A(\cdot)\in\{0,\ldots, N\}$ where $N$ serves
as the system-size parameter, while assumption~(2) ensures the reaction
system does not get absorbed at either boundary $\{0,N\}$. Assumption~(3)
is not essential, but with an explicit scaling of the rate $\ka(N)$ in
terms of $N$, the polynomial form of the rates $\la$ will make it easy
to also establish the scaling of the rates $\la(X(t))$ in terms of $N$
under a rescaling of the species amounts $X_A$ [we will occasionally
use the notation $\ka(N)$ for $\ka$ when awareness of dependence on $N$
is key].
Assumption~(4) is made to absorb the effect of the environment and other
species, the changes of which we will not keep track of explicitly.

Under these assumptions, our reaction network system can now be
expressed by
%
%e3 #&#
\begin{equation}
\label{reac1} \bigl\{aA+bB \mathop{\longrightarrow}^{\ka_\zeta^{ab}} (a+\zeta
)A+(b-\zeta)B\bigr\}_{ a\in\{0,\ldots,N\},b\in\{0,\ldots,N\}, \zeta\in\{
-a,\ldots,b\}}
\end{equation}
with reaction rates of the form $\la_\zeta^{ab}(x) = \ka_\zeta
^{ab}\cdot(x)_a(N-x)_b$.

Since the dynamics of the system depends on its overall drift, it will
be useful to distinguish a subset of reactions whose combined effect on
$\bE[X_A(t)]$ is zero, irrespective of the value of $X_A(t)$. In other
words, we will group reactions into a subset which contributes zero to
the drift (``balanced''), and the rest which are responsible for all of
the drift (``biased''). Note that the definition of balance below is
made for subsets of reactions---one cannot determine for a single
reaction on its own whether it is balanced or not---in order for a
reaction to be balanced it needs to belong to a balanced subset.

Let $\calI$ denote the set of all triples $(a,b,\zeta)$ for which a
reaction as written in \eqref{reac1} is present in the system.
A subset of reactions is defined as ``balanced'' $\calI^{\ba}\subset
\calI$ if for some fixed reactant amounts $a,b$, it satisfies
\[
\sum_{\zeta:(a,b,\zeta) \in\calI^{\ba}} \zeta\la_\zeta
^{ab}(x)=0 \qquad\forall x \quad\Longleftrightarrow \quad \sum
_{\zeta
\dvtx (a,b,\zeta)\in\calI^{\ba}} \zeta\ka_\zeta^{ab} = 0.
\]
A reaction $(a,b,\zeta)\in\calI$ that is part of some balanced subset
is called ``balanced'' and all the remaining reactions that are not
part of any balanced subset are called ``biased,'' $\calI^{\bi}=
\calI
-\calI^{\ba}$. Note that our notion of balance is very restrictive and
is not related to standard notions of chemical reactions.

For any balanced reaction $(a,b,\zeta)\in\calI^{\ba}$, there is
necessarily a reaction $(a,b,\zeta')\in\calI^{\ba}$ with $\zeta
,\zeta'$
having opposite signs (though not necessarily of the same size). Hence,
a reaction $(a,b,\zeta)\in\calI^{\ba}$ cannot have nontrivial rate at
the boundaries of the system: if $\la_\zeta^{a,b}(0)>0$ for some
$\zeta
>0$, then the balance condition would imply the existence of some
$\zeta
'<0$ for which $\la_{\zeta'}^{a,b}(0)>0$, which would violate
assumption~(1) by allowing $X_A$ to drop below $0$ upon a single further
$(a,b,\zeta')$ reaction. Consequently, the boundaries $0$ and $N$ are
absorbing for the balanced subsystem of reactions, and for all
$(a,b,\zeta')\in\calI^{\ba}$ we must have both $a>0$ and $b>0$.
Since assumption~(2) does not allow the boundary $\{0,N\}$ to be
absorbing for the full dynamics, this further implies that there is at
least one biased reaction $(a,b,\zeta)\in\calI^{\bi}$ with $\zeta>0$
and $\zeta\la_\zeta^{ab}(0)>0$, hence $a=0,b>0$; and there is at least
one biased reaction $(a,b,\zeta')\in\calI^{\bi}$ with $\zeta'<0$, and
$\zeta'\la_{\zeta'}^{a,0}(N)< 0$, hence $a>0,b=0$.

The continuous-time Markov jump process model for the reaction dynamics
can be expressed in terms of a set of Poisson processes under a random
time change.
Given a collection $\{Y_\zeta^{ab}\}_{(a,b,\zeta)\in\calI}$ of
independent unit-rate Poisson processes, the state of the system can be
expressed as a solution to the stochastic equation (see~\cite{KuAPP} or
\cite{BKPR} for details)
\begin{eqnarray*}
X_A(t)&=&X_A(0)+\sum
_{(a,b,\zeta)\in\calI} \zeta Y_\zeta^{ab} \biggl(\int
_0^t \la_\zeta^{ab}
\bigl(X_A(s)\bigr)\,ds \biggr)
\nonumber
\\
&=&X_A(0)+\sum_{(a,b,\zeta)\in\calI}\zeta
\hY_\zeta^{ab} \biggl(\int_0^t
\la_\zeta^{ab}\bigl(X_A(s)\bigr)\,ds \biggr)+ \int
_0^tF\bigl(X_A(s)\bigr)\,ds,
\end{eqnarray*}
where $\{\hY_\zeta^{ab}\}_{(a,b,\zeta)\in\calI}$ are centered Poisson
processes $\hY(\la t):=Y(\la t)-\la t$, and
\[
F(x)=\sum_{(a,b,\zeta)\in\calI} \zeta
\la_\zeta^{ab}(x)=\sum_{(a,b,\zeta)\in\calI^{\bi}} \zeta
\kappa_\zeta^{ab}(x)_a(N-x)_b.
\]

Since the capacity $N$ of the system may be arbitrarily large, we will
consider a ``standard'' rescaling of the system; see, for example,
\cite
{EK},\vadjust{\goodbreak} Chapter~11.2. Let $X_N(t)=N^{-1}X_A(t)$, then
%
%e4 #&#
\begin{eqnarray}
\label{resc1} X_N(t)& =&X_N(0) \nonumber\\
&&{}+ \sum
_{(\zeta,a,b) \in\calI} N^{-1}\zeta \hY _\zeta
^{ab} \biggl(N^{a+b}\ka_\zeta^{ab}\int
_0^t \bigl(X_N(s)
\bigr)_{a,N} \bigl(1-X_N(s)\bigr)_{b,N} \,ds \biggr)
\\
&&{}+ \int_0^tF_N\bigl(X_N(s)
\bigr) \,ds,\nonumber
\end{eqnarray}
where the local drift of the renormalized system is given by
\[
F_N(x)=\sum_{(\zeta,a,b) \in\calI^{\bi}}
N^{a+b-1}\zeta \ka _\zeta ^{ab} (x)_{a,N}(1-x)_{b,N}.
\]
The most important feature of the Markov jump process model is the
relationship of the variance to the drift. Note that we can write (\ref
{resc1}) as
\[
X_N(t)=X_N(0)+M_N(t)+\int
_0^t F_N\bigl(X_N(s)
\bigr)\,ds,
\]
where the second term from (\ref{resc1}), a weighted sum of
time-changed centered Poisson processes, is a martingale $M_N(t)$ whose
quadratic variation satisfies
\[
[M_N]_t=\sum_{(\zeta,a,b) \in\calI}
N^{-2}\zeta^2 Y_\zeta ^{ab}
\biggl(N^{a+b}\ka_\zeta^{ab} \int
_0^t \bigl(X_N(s)
\bigr)_{a,N}\bigl(1-X_N(s)\bigr)_{b,N} \,ds \biggr).
\]
Hence, if $\mathcal{F}_t=\sigma(X(s),0\le s\le t)$ denotes the natural
filtration of the process, then
\begin{eqnarray*}
&&\frac{d}{ds}\bE \bigl[X_N(s)|\mathcal{F}_t \bigr]
|_{s=t}\\
&&\qquad=\bE \bigl[F_N\bigl(X_N(t)\bigr)|
\mathcal{F}_t \bigr]
\\
&&\qquad=\sum_{(\zeta,a,b) \in\calI^{\bi}}
N^{a+b-1}\zeta \ka_\zeta^{ab} \bigl(X_N(t)
\bigr)_{a,N}\bigl(1-X_N(t)\bigr)_{b,N},
\\
&&\frac{d}{ds}\bE \bigl[\bigl(X_N(s)-\bE\bigl[X_N(s)
\bigr]\bigr)^2|\mathcal{F}_t \bigr] |_{s=t}\\
&&\qquad=
\frac{d}{ds}\bE \bigl[[M_N]_s|\mathcal{F}_t
\bigr] |_{s=t}\\
 &&\qquad= \sum_{(\zeta,a,b) \in\calI^{\ba}\cup\calI^{\bi}}
N^{a+b-2}\zeta^2 \ka_\zeta^{ab}
\bigl(X_N(t)\bigr)_{a,N}\bigl(1-X_N(t)
\bigr)_{b,N}.
\end{eqnarray*}
Recall that the reaction rates $\ka^{ab}_\zeta=\ka^{ab}_\zeta(N)$ also
depend on the scaling parameter~$N$. The standard scaling for a
reaction constant is $\ka^{ab}_\zeta(N)=\tka^{ab}_\zeta N^{1-(a+b)}$
for some $N$-independent constant $\tka^{ab}_\zeta$. However,
regardless of the chosen scaling of $\ka^{ab}_\zeta$, for biased
reactions $\calI^{\bi}$ the order of magnitude for each summand in the
infinitesimal variance $\frac{d}{ds}\bE [[M_N]_s|\mathcal
{F}_t
] |_{s=t}$ is $N^{-1}$ times smaller than the corresponding summand
in the infinitesimal drift $\bE [F_N(X_N(s))|\mathcal{F}_t
]
|_{s=t}$. This constrains the possible limiting dynamics of $X_N$.
Suppose the scaling of the rates is $\ka^{ab}_\zeta=N^{1-(a+b)} \tka
^{ab}_\zeta$, and note that then $F_N(x) \to\sum_{(\zeta,a,b) \in
\calI
^{\bi}} \zeta \tka_\zeta^{ab}x^a(1-x)^b$ uniformly for $x \in[0,1]$.
As established in \cite{Ku77}, in the limit as $N\to\infty$ the drift
overpowers the noise and, provided $X_N(0)\Rightarrow x(0)$, the
renormalized process $(X_N(t),t\ge0)$ converges in distribution (in
the Skorokhod topology of cadlag paths) to a solution $(x(t),t\ge0)$
of the ordinary differential equation
%
%e5 #&#
\begin{equation}
\label{ode} x(t)=x(0)+\int_0^t \sum
_{(\zeta,a,b) \in\calI^{\bi}} \zeta \tka _\zeta ^{ab}x(s)^a
\bigl(1-x(s)\bigr)^b \,ds.
\end{equation}
In fact, if the scaling of the reaction constants $\ka^{ab}_\zeta$ is
not standard, but is consistent for both balanced and biased reactions
in terms of the polynomial order of the rate function $\la^{ab}_\zeta$,
then the same deterministic limit is obtained under an appropriate time
rescaling.

The only way to get a stochastic limiting object for $X_N$ is for at
least one subset of balanced reactions to have a rate constant with a
different scaling in $N$. This different scaling needs to be such that
the noise term due to this subset of reactions will be of the same
order of magnitude as the overall drift from the biased reactions. This
would require a specific separation of time-scales for balanced versus
biased reactions. Although we do not exclude this possibility from our
analysis (see definition of $\ep_A$ at the end of this section), our
emphasis in this paper is on separating the time-scales in terms of
contribution of an additional source of noise, and its ability to
produce nontrivial random limiting objects for $X_N$.

%s2.2 #&#
\subsection{Stochastic model for resampling, branching or
splitting}\label{subsec:splitmodel}

We now introduce the additional mechanism in the system that describes
changes to species amounts due to the effect of splitting, branching or
resampling, which also effects the species count. For intracellular
molecular populations, our first model of splitting was motivated by a
simple double-then-divide principle: the cell will first double in size
by replicating its constituent molecular species, and then allocate
approximately one half of this doubled material into each daughter
cell---the allocation mechanism is not perfect and will make random
error from the original (undoubled) amount. For genetic populations,
common models for resampling follow the Wright--Fisher or the Moran
neutral reproduction law: each individual of the offspring population
chooses at random from the diploid version of the current population's
genes what to inherit---the resampling mechanism is such that an allele
of one type in one generation may at random be replaced in the
subsequent generation by an allele of the other type. These two are
both examples of a general mechanism with
the following key properties that we assume for splitting/resampling:
\begin{longlist}[(5)]
\item[(5)] The splitting/resampling occurs at rate $\gamma(x,N)$ that
depends on: the current state $X_A=x$ of the system and the scaling
parameter $N$; conditional on $X_A=x$ it is independent of reactions.
\item[(6)] The change in the species amount $X_A$ due to a
splitting/resampling event has the distribution $p_{x,y}=\bP
[X_A(t)=y|X_A(t-)=x]$ that have absorbing boundaries $p_{0,0}=1,
p_{N,N}=1$, and that are unbiased
\[
\mu_N(x)=\sum_yyp_{x,y}=x\qquad
\forall x \in\{0,1,\ldots,N\}.
\]
\end{longlist}
We also assume, for some of our results, that the rate $\gamma(x,N)$
and distribution $\{p_{x,y}\}$ are such that:
\begin{enumerate}[(${7}^*$)]
\item[(${7}^*$)] The change sizes are asymptotically uniformly bounded,\vspace*{3pt}
\begin{enumerate}[($7^*.\mathrm{a}$)]
\item[($7^*.\mathrm{a}$)]\quad
$\displaystyle\forall\Delta>0\qquad \sup_x
\gamma(x,N) \sum_{y:N^{-1}|y-x|\ge\Delta}p_{x,y}\to0 \qquad\mbox{as }
N\to\infty,$\vspace*{3pt}
\end{enumerate}
and the change size variance $ \sigma^2_N(x)=\sum_{y}(y-x)^2p_{x,y} $
is asymptotically given by
\begin{enumerate}[($7^*.\mathrm{b}$)]\vspace*{3pt}
\item[($7^*.\mathrm{b}$)]\quad
$\displaystyle\sup_x \bigl|
\gamma(x,N)N^{-2}\sigma ^2_N(x) -
\tgamma^2\tsig^2\bigl(N^{-1}x\bigr)\bigr |\to0\qquad
\mbox{as } N\to \infty$\vspace*{3pt}
\end{enumerate}
%
%yyy
for some constant $\tgamma>0$ and function $\tsig(\cdot)$ that are
independent of $N$, and such that $x\mapsto\tsig^2(x)$ is continuous
with $\tsig^2(x)>0, \forall x\in(0,1)$ and $\tsig^2(0)=\tsig^2(1)=0$.
\end{enumerate}
%
%Let $\mu_N(x)=\sum_{y} yp_{x,y}$ and $\sigma^2_N(x)=
%splitting/resampling ditribution.
Unbiasedness in assumption~(6) could be replaced by an ``asymptotic
unbiasedness'' assumption $N^{-1} |\mu_N(x)-x |\to0 \mbox{
as }
N\to\infty$, but for the sake of simplicity we assume $\mu_N(x)=x$.
Absorption in assumption~(6) implies splitting is noiseless on the
boundaries regardless of its time-scale. When the additional
assumption~(${7}^*$) holds (as we will assume for our results in
Section~\ref{subsec:qualbeh}), the splitting mechanism contributes
diffusively to the limit of the renormalized species count~$X_N$.
However, we will also examine the case when the rate of the splitting
mechanism is on a slower time-scale (in Section~\ref{subsec:bistable}),
as well as the case when it is on a faster time-scale (in Section~\ref{subsec:switch}).
The condition that $\tsig^2$ has boundary values
$\tsig^2(0)=\tsig^2(1)=0$ is natural given that any splitting or
resampling mechanism should absorb at the boundaries as indicated by
$p_{0,0}=p_{N,N}=1$.

\begin{example*}[(HG)]
One example of a splitting mechanism would be to
completely randomly reallocate the doubled content of a parent cell
into daughter cells. If the initial content\vadjust{\goodbreak} is $(X_A, X_B)=(x,N-x)$,
and the doubled content $(2x, 2(N-x))$ is partitioned in a single swoop
(draw without replacement) into two sets of $N$ molecules (one for each
daughter cell), then the content in each daughter cell has the
hypergeometric distribution (below we keep track of an arbitrarily
chosen single lineage)
\begin{eqnarray}
p_{x,y}=\bP\bigl[X_A(t)=y|X_A(t-)=x\bigr] =
\frac{{2x\choose y}
{2N-2x \choose N-y}}{{2N\choose N}},\nonumber\\
  \eqntext{0 \lor(2x-N) \leq y \leq2x\land N.}
\end{eqnarray}
The change in the species count is clearly unbiased $\mu_N(x)=\sum_{y=0
\lor N-2x}^{2x\land N} yp_{x,y}=x$, with variance
\[
\sigma^2_N(x)=\sum_{y=0\lor N-2x}^{2x\land N}(y-x)^2
p_{x,y}=\frac{N
2x(2N-2x)}{4N^2} \biggl(1-\frac{N-1}{2N-1} \biggr)=
\frac{x(N-x)}{2N-1}.
\]
Then assumption (${7}^*$) will hold if $\gamma(x,N)=\tgamma^2 N$ and
$\tsig^2(x)=\frac{1}2x(1-x)$, since for (${7}^*$.a) we have
\begin{eqnarray*}
\sup_x \bigl|\gamma(x,N)N^{-2}\sigma^2_N(x)-
\tgamma^2 \tsig ^2\bigl(N^{-1}x\bigr) \bigr|&=&
\tgamma^2\sup_x \biggl|\frac{x(N-x)}{N(2N-1)}-
\frac{1}2\frac{x}{N} \biggl(1-\frac{x}{N} \biggr) \biggr|\\
&\to&0
\end{eqnarray*}
and using tail bounds for the hypergeometric distribution \cite{Ch}
$\sum_{y=x+N\Delta}^N p_{x,y}\le e^{-2\Delta^2N}$ independently of $x$,
and for (${7}^*$.b) we have
\[
\gamma(x,N) \sup_x\sum_{y:|y-x|\ge N\Delta}p_{x,y}
\le2 \tgamma ^2 Ne^{-2\Delta^2N}\to0.
\]
\end{example*}
\begin{example*}[(Bin)]
Another example would be to sample with
replacement from the population in which each offspring picks its type
randomly from any individual in the parent generation. If the initial
count is $X_A=x$, then the count in the next generation has the
binomial distribution
\[
p_{x,y}=\bP\bigl[X_A(t)=y|X_A(t-)=x\bigr] =
{\pmatrix{N\cr y}} \biggl(\frac
{x}{N} \biggr)^y \biggl(1-
\frac{x}{N} \biggr)^{N-y} , \qquad 0\leq y \leq N.
\]
This form of resampling is used in (the haploid version of) the
Wright--Fisher model for genetic drift (e.g., \cite{DurDNA} Section~1.2). It is also used as the prototype of a splitting mechanism of
simple ``independent segregation'' of division of cells \cite{Pa2}.
This distribution is again unbiased, and assumption (${7}^*$) will
hold if $\gamma(x,N)=\frac{1}2\tgamma^2 N$ for some constant $\tgamma
^2>0$. Using similar arguments as above, it is then easy to show that
both (${7}^*$.a) and (${7}^*$.b) will hold with $\tsig
^2(x)=\frac{1}{2}x(1-x)$.
\end{example*}

\begin{example*}[(Bern)]
Finally, the simplest example of a
splitting/resampling mechanism is to have a\vadjust{\goodbreak} single amount error in the
daughter cell (or the next generation), and to have the rate at which
the error occurs be proportional to both the current amount $X_A=x$ and
the amount of $X_B=N-x$. Errors from imperfect division will result in
$\pm$ change with equal probability
\[
p_{x,x-1}=p_{x,x+1}=1/2.
\]
This distribution is clearly unbiased, and assumption (${7}^*$) will
hold if the rate of error occurrences is $\gamma(x,N)=\frac{1}2\tgamma^2
N^2\frac{x}N(1-\frac{x}N)$ for some $\tgamma^2>0$, with the limiting
variance $\tsig^2(x)=\frac{1}2x(1-x)$. \\
This form of resampling is used in the Moran model for genetic drift
(e.g., \cite{DurDNA} Section~1.5).
It is also used in \cite{Pa2} as an example of an ``ordered
segregation'' splitting mechanism for cell division (self volume
exclusion partitioning error, \cite{Pa2} Supporting Information). In
a cellular system it could also be described as a set of balanced
reactions $A+B\to2A, A+B\to2B$ with mass-action dynamics and
appropriately scaled rate constants.
\end{example*}

We note that, from the perspective of limiting results, the differences
in the specific details of the mechanism will not be important. The
only feature of relevance will be the order of magnitude of the
prelimiting rate $\gamma(x,N)$ and the form of the limiting variance
$\tsig^2(x)$. There are many other types of splitting, branching or
resampling mechanisms, yielding a different form for the limiting
variance. They are easy to construct in case of small changes that
result in single count errors, via a range of birth--death probability
distributions. We shall see, in both Section~\ref{sec:bistable} and
Section~\ref{sec:switching}, how the actual form for the variance
$\tsig
^2(x)$ affects the qualitative behavior of the limit of the
renormalized process.

The changes due to this additional mechanism can also be expressed in
terms of a Poisson processes under a random time change. Let $Y_\gamma$
be a counting process with state-dependent rate $\gamma(x,N)$, and $\{
Z(x,s)\}_{0\le x\le N}$ be independent random variables with
probability distribution $p_{x,\cdot}$ for any $s\ge0$. A change due
to splitting or resampling can be represented as a stochastic integral
$\int_0^t (Z(X(s-),s)-X(s-))\,dY_\gamma(s)$. The evolution in species
count due to both reaction dynamics and splitting is
\begin{eqnarray*}
X_A(t)&=&X_A(0)+\sum
_{(\zeta,a,b)\in I}\zeta\hY_\zeta^{ab} \biggl(\int
_0^t \la_\zeta^{ab}
\bigl(X_A(s)\bigr)\,ds \biggr)+ \int_0^t
F\bigl(X_A(s)\bigr)\,ds
\\
&&{} +\int_0^t \bigl(Z\bigl(X_A(s-),s
\bigr)-X_A(s-)\bigr)\,dY_\gamma(s);
\end{eqnarray*}
hence for the rescaled system $X_N=N^{-1}X_A$ we have
%
%e6 #&#
\begin{eqnarray}
\label{resc2} X_N(t) &=& X_N(0) \nonumber\\
&&{}+ N^{-1} \sum
_{(\zeta,a,b) \in\calI} \zeta \hY _\zeta^{ab}
\biggl(N^{a+b}\ka_\zeta^{ab}\int_0^t
\bigl(X_N(s)\bigr)_{a,N} \bigl(1-X_N(s)
\bigr)_{b,N} \,ds \biggr)\nonumber\\
&&{}+ \int_0^tF_N
\bigl(X(s)\bigr)\,ds
\\
&&{} +\int_0^t \bigl(N^{-1}Z
\bigl(NX_N(s-),s\bigr)-X_N(s-) \bigr)\,d\hY _\gamma
(s)
\nonumber
\\
&&{} +\int_0^t \bigl(N^{-1}Z
\bigl(NX_N(s-),s\bigr)-X_N(s-) \bigr)\gamma
\bigl(NX_N(s),N\bigr) \,ds\nonumber
\end{eqnarray}\vspace*{-12pt}
\def\@eqnnum{\hb@xt@.01\p@{}\rlap{\normalfont\normalcolor
              \hskip -\displayindent\hskip -\displaywidth\theequation}}
\def\theequation{6$'$}
\begin{eqnarray}
\hspace*{17pt}&=& X_N(0) + M_{N,\gamma}(t)+ \int_0^tF_N
\bigl(X(s)\bigr)\,ds
\nonumber
\\[-8pt]
\\[-8pt]
\nonumber
& &{}+\int_0^tN^{-1} \bigl(Z
\bigl(NX_N(s-),s\bigr)-NX_N(s-) \bigr)\gamma
\bigl(NX_N(s),N\bigr) \,ds.
\nonumber
\end{eqnarray}
We still have
%
%e7 #&#
\renewcommand{\theequation}{\arabic{equation}}
\setcounter{equation}{6}
\begin{equation}
\label{FN} F_N(x)=\sum_{(\zeta,a,b) \in\calI^{\bi}}
N^{a+b-1}\zeta \ka _\zeta ^{ab}(N)
(x)_{a,N}(1-x)_{b,N},
\end{equation}
but now $M_{N,\gamma}$ denotes the martingale formed by the second and
fourth summand in (\ref{resc2})
whose quadratic variation is
%
%e8 #&#
\begin{eqnarray}
\label{qvar2} &&[M_{N,\gamma}]_t\nonumber\\
&&\qquad=\sum
_{(\zeta,a,b) \in\calI} N^{-2}\zeta^2 Y_\zeta
^{ab} \biggl(N^{a+b}\ka_\zeta^{ab}(N) \int
_0^t \bigl(X_N(s)
\bigr)_{a,N}\bigl(1-X_N(s)\bigr)_{b,N} \,ds \biggr)
\\
&&\qquad\quad{} +\int_0^t N^{-2} \bigl(Z
\bigl(NX_N(s-),s\bigr)-NX_N(s-) \bigr)^2\,dY_\gamma(s).\nonumber
\end{eqnarray}
Note that since the two mechanisms are driven by independent Poisson
processes, there is no quadratic covariation contribution. Since $\bE
_{p_{x,\cdot}} [Z(x,s)-x ]=0$ for all $s\ge0$ and $x\in\{
0,\ldots
,N\}$, the infinitesimal mean still satisfies
%
%e9 #&#
\begin{eqnarray}
\label{infmean} &&\frac{d}{ds}\bE \bigl[X_N(s)|
\mathcal{F}_t \bigr] \Big|_{s=t}\nonumber\\
&&\qquad=\bE \bigl[F_N
\bigl(X_N(t)\bigr)|\mathcal{F}_t \bigr]
\\
&&\qquad=\sum
_{(\zeta,a,b) \in\calI^{\bi}} N^{a+b-1}\zeta \ka_\zeta^{ab}(N)
 \bigl(X_N(t)\bigr)_{a,N}\bigl(1-X_N(t)
\bigr)_{b,N};\nonumber
\end{eqnarray}
on the other hand, the infinitesimal variance now satisfies
%
%e10 #&#
\begin{eqnarray}
\label{infvariance} &&\frac{d}{ds}\bE \bigl[\bigl(X_N(s)-\bE
\bigl[X_N(s)\bigr]\bigr)^2|\mathcal{F}_t
\bigr] \Big|_{s=t}\nonumber\\
&&\qquad=\frac{d}{ds}\bE \bigl[[M_N]_s|
\mathcal{F}_t \bigr] \Big|_{s=t}
\nonumber
\\[-8pt]
\\[-8pt]
\nonumber
&&\qquad= \sum_{(\zeta,a,b) \in\calI^{\ba}\cup\calI
^{\bi}} N^{a+b-2}\zeta^2
\ka_\zeta^{ab}(N) \bigl(X_N(s)
\bigr)_{a,N}\bigl(1-X_N(s)\bigr)_{b,N} \\
&&\hspace*{96pt}{}+\gamma
\bigl(NX_N(s),N\bigr)N^{-2} \sigma_N^2
\bigl(NX_N(s)\bigr).\nonumber
\end{eqnarray}
%

%s2.3 #&#
\subsection{Possible qualitative behaviors}\label{subsec:qualbeh}

In order to determine the role that the rate of the
splitting/resampling mechanism may play, we first establish the
possible behavior of the system when $N$ is large. The decisive
quantity for the qualitative behavior of the system is
%2. \ep_A as defined in (11) is a "post-limiting" quantity, so I felt
%we should avoid writing \ve_A \to\infty, \ve_A \approx0, \ve_A
%If you'd rather keep what we have and define a "pre-limit" \ve_A(N) or
%something like that, then by all means feel free to go ahead and
%change it back.
%e11 #&#
\begin{equation}
\label{epsA} \ve_A:= \lim_{N\to\infty}
\ve_A(N), \qquad\ve_A(N):= \frac{c_{\sigma
^2}(N)}{c_{\mu}(N)},
\end{equation}
where
%
%e12 #&#
\begin{equation}
\label{ordervar} c_{\sigma^2}(N):=\sum_{(a,b,\zeta)\in\calI^{\ba}}N^{a+b-2}
\ka _\zeta ^{ab}(N) +\sup_{x\in[0,1]}
\gamma(Nx,N)N^{-2}\sigma^2_N(Nx)
\end{equation}
and
%
%e13 #&#
\begin{equation}
\label{ordermean} c_{\mu}(N):= \sum_{(a,b,\zeta)\in\calI^{\bi}}N^{a+b-1}
\ka_\zeta^{ab}(N);
\end{equation}
$\ep_A$ relates the magnitude of the variance due to the splitting
mechanism (or possibly a faster set of balanced reactions) to the
magnitude of the drift due to reaction dynamics. If they are of the
same order of magnitude, then the rescaled process will converge to a diffusion.
In other words, if $\ve_A\in(0,\infty)$, then we can assume (by
rescaling time as necessary) that both scaling constants (\ref
{ordervar}) and (\ref{ordermean}) satisfy $\tilde{c}_{\sigma^2}=\lim_{N\to\infty} c_{\sigma^2}(N)\in(0,\infty), \tilde{c}_\mu=\lim_{N\to
\infty} c_{\mu}(N)\in(0,\infty)$ and $\tilde{c}_{\sigma^2}=\ve
_A\tilde
{c}_{\mu}$. If assumption (${7}^*$) is satisfied, the noise of the
splitting mechanism is such that the limiting behavior of the system is
diffusive, instead of being deterministic, as in~(\ref{ode}) when only
reactions are present.

%pr2.1 #&#
\begin{proposition}\label{sdethm} If $\ve_A\in(0,\infty)$,
assumption (${7}^*$) holds for $X_A=NX_N$, and $X_N(0)\Rightarrow\tX
(0)\in[0,1]$, then $X_N\Rightarrow\tX$ as $N\to\infty$ in distribution
on the Skorokhod space of cadlag paths on $[0,1]$, where $\tX$ is a
diffusion with drift and diffusion coefficients given by
%
%e14 #&#
\begin{eqnarray}
\label{diffcoeff} \tphi(x)&=& \sum_{(a,b,\zeta)\in\calI^{\bi}} \zeta \tka
^{ab}_{\zeta, \mu} x^a(1-x)^b,
\nonumber
\\[-8pt]
\\[-8pt]
\nonumber
\tila(x)&=&
\sum_{(a,b,\zeta
)\in\calI^{\ba}} \zeta^2 \tka^{ab}_{\zeta, \sigma^2}
x^a(1-x)^b+\tgamma^2\tsig^2(x),
\end{eqnarray}
where for each $(a,b,\zeta)\in\calI^{\bi}$
\[
\tka^{ab}_{\zeta,\mu} = \lim_{N\to\infty}
N^{a+b-1} \ka_\zeta^{ab}(N)
\]
for each $(a,b,\zeta)\in\calI^{\ba}$
\[
\tka^{ab}_{\zeta,\sigma^2}=\lim_{N\to\infty}
N^{a+b-2} \ka_\zeta^{ab}(N)
\]
and for some $\tgamma^2>0$
\[
\tgamma^2\tsig^2(x)=\lim_{N\to\infty}
\gamma(Nx,N)N^{-2}\sigma^2_N(Nx).
\]
If all reaction rates have standard scaling $\ka_\zeta^{ab}=\tka
_\zeta
^{ab}N^{1-(a+b)}$, then $\tka^{ab}_{\zeta,\sigma^2}=0$ and $\tila
(x)=\tgamma^2\tsig^2(x)$.
\end{proposition}

\begin{pf} This is a direct consequence of standard theorems for
convergence of Markov processes to a diffusion (see, e.g., \cite{DurSC}
Section~8.7) based on locally uniform convergence of the infinitesimal
mean and variance to the limiting drift and diffusion coefficients,
respectively, and convergence of jumps so that they disappear in the
limit. Recall that the infinitesimal mean of the rescaled process $X_N$
from (\ref{resc2}$'$) is given by (\ref{infmean}) and its infinitesimal
variance by (\ref{infvariance}). Since the process takes values in
$[0,1]$, we can check convergence uniformly on the whole space, and
moreover $M_{N,\gamma}(t)=X_N(t)-\bE[X_N(t)]$, whose quadratic
variation is given in (\ref{qvar2}), is then a square integrable
martingale. For the contributions by the splitting mechanism, the
convergence of the infinitesimal mean and variance, as well as the
control of the jumps, are easy to check from the three requirements on
the splitting mechanism made in assumptions (6) and (${7}^*$). For the
contributions by the reaction dynamics the convergence of the
infinitesimal mean and variance, and the control over jumps, follow
from the scaling properties of the counting processes used in their
representation and from the fact that the rates for these counting
processes are Lipschitz and bounded. These same conditions have been
checked, in the case when reaction rates have a more general form, for
law of large numbers and central limit theorem results for rescaled
population-dependent Markov processes \cite{Ku77}. Alternatively, one
could also check that the Markov process $X_N$ satisfies all the
conditions required for convergence of more general Markov jump
processes to a diffusion as stated in Theorem~2.11 of \cite{Ku70} and
Theorem~3.1 of \cite{Ku71}.
The only thing left to check is whether a diffusion with coefficients
as given exists and is unique in law. This follows easily from the fact
that the contributions to $\tila(x)$ and $\tphi(x)$ from reaction rates
are polynomial, and we have assumed that $\tsig^2(x)$ is Lipschitz.
\end{pf}

A diffusion may or may not hit its boundary points, but it never spends
a disproportionate amount of time at any point in its range, including
the boundaries, unless they are absorbing. Hence, we really need to
consider the behavior of the process when either $\ve_A\to0$ or $\ve
_A\to\infty$ (as a function of an additional asymptotic parameter which
will be discussed below in Section~\ref{subsec:bistable}). The only
remark we make when $\ve_A$ remains bounded away from $0$ and $\infty$
is that the behavior\vadjust{\goodbreak} of $\tX$ at the boundary $\{0,1\}$ depends on the
form for the limiting variance of the splitting mechanism. As a
consequence of assumption~(2), and of the properties of the splitting
variance at $\{0,1\}$, we are only guaranteed that $\tphi(0)>0, \tphi
(1)<0$ and $\tila(0)=\tila(1)=0$. Hence, $\{0,1\}$ are neither
absorbing nor natural, but it remains to determine whether they are
entrance or regular boundary points. Further conditions on the reaction
and splitting mechanisms for reaching the boundary (i.e., for $\{0,1\}$
to be regular boundary points) are guaranteed by interpreting Feller's
test for explosion; see, for example, \cite{DurSC}, Section~6.2. or
\cite{KT}, Section~15.6.

The diffusive case $\ve_A\in(0,\infty)$ separates two other types of
behavior. When $\ve_A\approx0$ and $\ve_A\approx\infty$, the rate of
splitting is either slower or faster, respectively, than prescribed by
assumption (${7}^*$). Both cases lead to behavior which exhibits a
type of stochastic bistability, in which the system spends almost all
of its time at two points, or very near them. This bistability is, in
the two cases $\ve_A\approx0$ and $\ve_A\approx\infty$, caused by
completely different effects of the two stochastic mechanisms in our
model, which we investigate separately in the next two
sections.\vspace*{-2pt}

%s3 #&#
\section{Bistable behavior from slow splitting}\label{sec:bistable}

Let us consider the case $\ve_A\approx0$, and assume that time has
been rescaled so that $\tilde{c}_\mu=\lim_{N\to\infty}c_\mu(N)
\in
(0,\infty)$ and $\tilde{c}_{\sigma^2}=\lim_{N\to\infty}c_{\sigma
^2}(N)\approx0$.
In modeling this is a relatively conventional scaling, in which a small
amount of noise (from balanced reactions and splitting) will affect the
predominantly deterministic behavior due to drift (of biased reactions).
A~precise statement of this depends on how fast $\ve_A(N)=\frac
{c_{\sigma^2}(N)}{c_{\mu}(N)}$ approaches $0$ as a function of $N$, and
we examine it more carefully by first introducing a separate
perturbation parameter $\ve$ and then relating it to the scaling
parameter $N$.\vspace*{-2pt}

%s3.1 #&#
\subsection{Small diffusive noise effects}\label{subsec:bistable}

The simplest way to model small diffusive effects is with an enforced
separation of time-scales between reactions and splitting using a
perturbation parameter. Suppose all the reaction constants $\ka
^{ab}_\zeta(N)$ depend only on the scale of the system $N$ and have the
standard scaling $\ka^{ab}_\zeta=\tka^{ab}_\zeta N^{1-(a+b)}$ for some
constants $\tka^{ab}_\zeta$. Suppose the splitting rate, in addition to
$N$, also depends on a small parameter $\ve>0$, so that the splitting
rate is $\gamma(x,N,\ve) =\ve^2\gamma(x,N)$ where $\gamma(x,N)$
satisfies assumption (${7}^*$). The fact that the splitting rate is
slower than diffusive is expressed in terms of the fact that we will
consider the behavior of the system as $\ve\to0$. In this case the
quantitiy $\ve_A$ defined in (\ref{epsA}) is just a constant multiple
of $\ve^2$
\begin{eqnarray*}
\ep_A &:=& \lim_{N\to\infty}
\frac{\sum_{(a,b,\zeta)\in I^{\ba
}}N^{a+b-2}\ka_\zeta^{ab}(N) +\ve^2\sup_{x\in[0,1]}N^{-2}\gamma
(Nx,N)\sigma^2_N(Nx)}{\sum_{(a,b,\zeta)\in I^{\bi}}N^{a+b-1}\ka
_\zeta
^{ab}(N)}\\
&=&\frac{\ve^2 \tilde{c}_{\sigma^2}}{\tilde{c}_\mu},
\end{eqnarray*}
where $\tilde{c}_{\sigma^2}=\tgamma^2\sup_{x\in[0,1]} \tsig^2(x)$ and
$\tilde{c}_\mu=\sum_{(a,b,\zeta)\in
I^{\bi}}\tka^{ab}_{\zeta,\mu}$.\vadjust{\goodbreak}

We could also assume the rates of balanced reactions depend on the
additional parameter $\ve^2$, in the sense that $\ka^{ab}_\zeta=\ve
^2\tka^{ab}_\zeta N^{2-(a+b)}$ for $(a,b,\zeta)\in\calI^{\ba}$. In
this case
\[
\tilde{c}_{\sigma^2}=\sum_{(a,b,\zeta)\in I^{\ba}}\tka
^{ab}_{\zeta,\mu
}+\tgamma^2\sup_{x\in[0,1]}
\tsig^2(x).
\]
However, if we make no special separation in the way balanced and
biased reactions are scaled, then the assumption of standard scaling
$\ka^{ab}_\zeta=\tka^{ab}_\zeta N^{1-(a+b)}$ implies that this is only
possible if the parameter $\ve$ satisfies $\ve^2=N^{-1} $, on which we
remark further in the next subsection.

By Proposition~\ref{sdethm}, for any fixed $\ve>0$, the process
obtained in the limit $X_N\Rightarrow\tX_\ep$ is a diffusion with
coefficients $\tphi(x)$ as in (\ref{diffcoeff}) and $\tila_\ve
(x)=\ve
^2\tgamma^2\tsig^2(x)$ (we will use the subscript $\ve$ in the notation
of the limiting diffusion to stress its dependence on the small
parameter $\ve$). $\tX_\ve$ is a solution of the stochastic
differential equation
%
%e15 #&#
\begin{equation}
\label{sdeLD} d\tX_\ve(t) = \tphi\bigl(\tX_\ve(t)\bigr)\,dt+
\ve\tgamma\tsig\bigl(\tX_\ve (t)\bigr)\,dB(t), \qquad\tX_\ve
\in[0,1],
\end{equation}
where $B$ is a standard Brownian motion, a classical case of a
diffusion with small diffusion coefficient.

For many such diffusions $\ep\approx0$ will have little qualitative
effect relative to $\ep=0$; however, suppose that $\tphi$ has two
stable and one unstable equilibria, and thus the potential $\Phi$
defined by $\Phi= -\int\tphi$ is a double-well potential. Since
$\tphi
$ is a polynomial, this is an assumption on the number and type of
zeros of $\tphi$. Explicitly, we will assume that
%
%e16 #&#
\begin{eqnarray}
\label{doublewell} \exists 0<x_1<x_2<x_3<1
\dvtx \tphi(x_i)=0,
\nonumber
\\[-8pt]
\\[-8pt]
\eqntext{i=1,2,3 \mbox{ and } \tphi '(x_1)<0,
\tphi'(x_2)>0,\tphi'(x_3)<0.}
\end{eqnarray}
Recall also that assumption~(2) implies that at the boundaries we have
$\tphi(0)>0, \tphi(1)<0$. As a consequence, $\tX_\ve$ is a process
whose mean behavior involves monotone convergence to one of two stable
equilibria (determined by the initial conditions), but where the small
amount of noise allows the process to switch from one equilibrium to
the other, creating a bistable system. Precise statements of this
behavior are described by Freidlin--Wentzell theory for random
perturbations of dynamical systems by diffusive noise, \cite{FW}, which
can also cover processes with metastability, \cite{GOV}.
We will follow closely the notation of \cite{GOV}, as these results
apply most directly to $\tX_\ve$. We first need a transformation to
handle the state dependence $\tsig^2(x)$ of the diffusion coefficient,
easily done using \cite{DZ}, Section~5.6. or~\cite{OV} Section~2.5.

For $\tX_\ep$ satisfying (\ref{sdeLD}), large deviation theory for
Gaussian perturbations of dynamical systems, Dembo and Zeitouni (\cite
{DZ} Theorem~5.6.7 and Exercise 5.6.25), state that deviations of $\tX
_\ep$ away from an $\ve$-sized neighborhood of $x_1$ and $x_3$ are
characterized by the large deviation rate\vadjust{\goodbreak} function for $\tX_\ve$ given
by the quasipotential (with respect to $x_i$ and $x_2$)
\begin{eqnarray}
&&I_{x_i,x_2}(\tphi,\tgamma\tsig)\nonumber\\
&&\qquad:=\inf_{s>0} \inf
_\xi \biggl\{ \int_0^s L
\bigl(\xi(u),\xi'(u)\bigr)\,du \Big| \xi\in C^1\bigl([0,s]
\bigr), \xi(0)=x_i,\xi (s)=x_2 \biggr\},\nonumber\\
\eqntext{ i=1,3,}
%yyy - I replaced all of the \varphi= \vp
%with \xi
\end{eqnarray}
where $L$ is the action functional
\[
L\bigl(\xi,\xi'\bigr) = \biggl(\frac{\xi' - \tphi(\xi)}{\tgamma\tsig(\xi
)}
\biggr)^2.
\]
This identifies the most likely paths which leave a neighborhood of
$x_1$ or $x_3$, since every path between $x_1$ and $x_3$ of the
one-dimensional $\tX_\ep$ has to pass through $x_2$. We can write
$L(\xi
,\xi')$ in this form for all such paths because $\tX_\ep$ is
nonsingular away from the boundaries, that is, $\tsig^2(x)>c, \forall
x\in[x_1,x_3]$ for some $c>0$.
If the diffusion coefficient were constant $\tgamma\tsig\equiv1$,
then $L(\xi,\xi') = (\xi' - \tphi(\xi))^2$ and
$I_{x_i,x_2}(\tphi,1)$ would be simply a constant multiple of the
potential, $I_{x_i,x_2}(\tphi,1) = 2(\Phi(x_2)-\Phi(x_i))$, for
$i=1,3$. The quasipotential would be determined by the height of the
potential barrier which $\tX_\ep$ needs to overcome in order to pass
from one equilibrium to the basin of attraction of the other.

To solve the variational problem in our case, we can use a
transformation of the path space $\xi= g(\psi)$ to get an action
functional of the form $L(\xi,\xi') = (\psi' - \tphi(\psi))^2$, from
which we can deduce the explicit form of the rate function
$I_{x_i,x_2}$ for state-dependent $\tgamma\tsig(x)$. For any monotone
$C^1$ function $g$ which for all $s$ is surjective from $C^1([0,s])$ to
$C^1([0,s])$, %for any $\xi\in C^1([0,s])$, there exists $\psi\in
%C^1([0,s])$ such that $\xi=g(\psi)$,
we have
\begin{eqnarray*}
&&I_{x_i,x_2}(\tphi,\tgamma\tsig) = \inf_{s >0} \inf
_\psi \biggl\{ \int_0^s L
\bigl(g\bigl(\psi(u)\bigr), \bigl[g\bigl(\psi (u)\bigr)\bigr]'
\bigr)\,du \Big| \psi\in C^1\bigl([0,s]\bigr),\\
&& \hspace*{166pt}\psi(0)=g^{-1}(x_i),
\psi (s)=g^{-1}(x_2) \biggr\}.
\end{eqnarray*}
We take $g$ which satisfies the (autonomous) first-order ODE
$g'(y)=\tgamma\tsig(g(y))$,
so that
\[
L\bigl(g(\psi), \bigl[g(\psi)\bigr]'\bigr) = \biggl(
\frac{g'(\psi)\psi'- \tphi(g(\psi
))}{\tgamma\tsig(g(\psi))} \biggr)^2= \biggl(\psi'-
\frac{\tphi
(g(\psi
))}{\tgamma\tsig(g(\psi))} \biggr)^2.
\]
Note that $\tgamma\tsig(x)>0, \forall x\in(0,1)$ ensures that $g$ is in
fact strictly increasing on $(0,1)$. Let $h(x) = g^{-1}(x)$. Then, if
$\tphi$ is the vector field of a double-well potential, so is $\al$
defined as
$\al=\frac{\tphi\circ g}{\tgamma\tsig\circ g}$, for the following reasons.
Let $y_i= g^{-1}(x_i) = h(x_i)$; these will be the equilibria for $\al
$, since $\al(y_i) = \tphi(g(y_i))/\tgamma\tsig(g(y_i)) = \tphi
(x_i)/\tgamma\tsig(x_i) = 0$. As for their stability, we have
\begin{eqnarray*}
\al'(y_i) &=& \frac{\tphi'(g(y_i))g'(y_i)\tgamma\tsig(y_i)
- \tphi
(g(y_i)) \tgamma\tsig'(g(y_i))g'(y_i)}{\tgamma^2\tsig^2(g(y_i))} \\
&=&
\frac{\tphi'(g(y_i))g'(y_i)}{\tgamma\tsig(g(y_i))} = \tphi'\bigl(g(y_i)\bigr),
\end{eqnarray*}
where the first equality holds since $\tphi(g(y_i))=0$, and the second
by definition of $g$. Therefore, for each $i$, the stability of $x_i$
under the vector field $\tphi$ is the same as that of $y_i$ with vector
field $\al$; we may therefore define $A = -\int\al$ to be the
(double-well) potential associated with $\al$. Since $L(\xi,\xi')$ is
now in the form $L(g(\psi), [g(\psi)]') =  (\psi'- \al(g(\psi
)) )^2$, we can conclude that
%
%e17 #&#
\begin{equation}
\label{potdiff2} I_{x_i,x_2}(\tphi,\tgamma\tsig)= I_{y_i,y_2}(\al,1) =
2\bigl(A(y_2)-A(y_i)\bigr),\qquad  i=1,3.
\end{equation}

We can now interpret the results of \cite{GOV} to characterize the
behavior of the process $\tX_\ve$ [defined in \eqref{sdeLD}] as $\ve
\to0$.
Let $D_i$ denote basins of attraction for the deterministic process
(\ref{ode}) driven by the drift $\tphi$, that is, $D_1 = [0,x_2)\ni
x_1, D_2= \{x_2\}, D_3 = (x_2,1]\ni x_3$, and $B_c(x_i)$ denote closed
balls of radius $c>0$ around $x_1,x_3$ such that $B_c(x_1)\subset D_1,
B_c(x_3)\subset D_3$. If the wells of the transformed potential $A$ are
not at equal depth $A(y_1)\neq A(y_3)$, we will without loss of
generality assume $A(y_1)<A(y_3)$. Let
\[
T_\ve= \inf\bigl\{t>0\dvtx \tX_\ve(t) \in
B_c(x_1)\bigr\},\qquad \tT_\ve= \inf\bigl\{
t>T_\ve \dvtx \tX_\ve(t) \in B_c(x_3)
\bigr\}
\]
denote the first hitting time of the neighborhood of the stable
equilibrium with the deeper basin, and the subsequent first hitting
time of the neighborhood of the other stable equilibrium. Let $\beta
_\ve$
be the time-scale on which transitions from $D_3$ to the neighborhood
of $x_1$ happen, defined by $\bP[T_\ve> \beta_\ve| \tX_\ve
(0)=x_3] =
e^{-1}$, and $\tilde{\beta}_\ve$ the one on which the reverse transition
happen, defined by $\bP[\tT_\ve> \tilde{\beta}_\ve| \tX_\ve
(0)=x_1] = e^{-1}$.
The next result establishes that the transition from one stable
equilibrium to the other happens on a time-scale of order $O(e^{\ve
^{-2}(A(y_2)-A(y_i))})$ with $i=3$ and $i=1$, respectively, and that in
the limit as $\ve\to0$ the transition times have an exponential distribution.

%pr3.1 #&#
\begin{proposition}\label{LD1} If $\tphi$ satisfies (\ref{doublewell}),
then the transitions of $\tX_\ve$ from $D_3$ to $B_c(x_1)$ and from
$D_1$ to $B_c(x_3)$ satisfy:
\begin{eqnarray*}
\mathrm{(i)}&&\quad \lim_{\ve\to0} \bP\bigl[T_\ve>t
\beta_\ve| \tX_\ve(0)=x\in D_3\bigr] =
e^{-t}\qquad\forall t>0,\\
&&\quad \lim_{\ve\to0} \bP\bigl[
\tT_\ve>t \tilde{\beta }_\ve | \tX_\ve(0)=x\in
D_1\bigr] = e^{-t}\qquad\forall t>0;
\\
\mathrm{(ii)}&&\quad \lim_{\ve\to0} \ve^2 \ln\beta_\ve=
I_{x_3,x_2}(\tphi ,\tgamma \tsig)=2\bigl(A(y_2)-A(y_3)
\bigr),\\
&&\quad \lim_{\ve\to0} \ve^2 \ln\tilde{\beta
}_\ve= I_{x_1,x_2}(\tphi,\tgamma\tsig)=2\bigl(A(y_2)-A(y_1)
\bigr).
\end{eqnarray*}
\end{proposition}

\begin{pf} (i) is a restatement of Theorem~1 in \cite{GOV}. (ii)
follows from Theorem~4.2 of Chapter~4 in \cite{FW}, which states that
for any $\delta>0$, $\lim_{\ve\to0} \bP[|\ve^2 \ln T_\ve
-I_{x_3,x_2}(\tphi,\tgamma\tsig)|>\delta| \tX_\ve(0)=x_3]=0$ and
$\lim_{\ve\to0} \bP[|\ve^2 \ln\tT_\ve-I_{x_1,x_2}(\tphi,\tgamma
\tsig
)|>\delta| \tX_\ve(0)=x_1]=0$, and from our explicit calculation of
the value of the quasipotential in (\ref{potdiff2}).
\end{pf}

The following result characterizes the long-term behavior on the
natural time-scale (determined by $\beta_\ve$) for transition to the
stable point with the deeper basin. Let $R_\ve=e^{\ve^{-2}a}$ for some
$a\in(0,2(A(y_2)-A(y_3)))$, so that $R_\ve\to\infty$ while $R_\ve
/\beta
_\ve\to0$ as $\ve\to0$. Again following \cite{GOV}, define the
measure-valued process $(\nu_t^\ve)_{t\ge0}$ by
\[
\nu_t^\ve(f) = \frac{1}{R_\ve} \int
_{\beta_\ve t}^{\beta_\ve t
+R_\ve} f\bigl(\tX _\ve(s)\bigr)\,ds
\]
for any (bounded) continuous function $f$ on $[0,1]$. The measure $\nu
_t^\ve$ approximates the law for the location of $\tX_\ve(T)$ on the
time-scale $T=\beta_\ve t$.

Note that if $A(y_1)<A(y_3)$, then the results of (ii) imply that
$\tilde{\beta}_\ve/\beta_\ve\to\infty$ as $\ve\to0$, so that
$\inf_{x\in
D_1}\bP[\tT_\ve/\beta_\ve>t| \tX_\ve(0)=x]\to1$. Hence, in this case
metastability is characterized by the fact that the transitions into
the deeper well are on an exponentially faster time-scale, relative to
which the transitions back into the less deep well will not be noticed.
Let $\bP_x[ \cdot]$ denote $\bP[ \cdot |\tX_\ve(0)=x]$.

%pr3.2 #&#
\begin{proposition}\label{LD2}
For each $x \in D_3$, continuous function $f$ on $[0,1]$, and $\delta
>0$ we have
\begin{eqnarray*}
\lim_{\ve\to0}\bP_x \Bigl[\sup_{s\in[0, {(T_\ve-3R_\ve)
}/{R_\ve
}]}\bigl|
\nu_t^\ve(f)-f(x_3)\bigr|>\delta \Bigr]&=&0, \\
\lim
_{\ve\to0}\bP \Bigl[\sup_{s\in[{T_\ve}/{\beta_\ve},{(\tT_\ve-3R_\ve)}/{R_\ve}]} \bigl|
\nu_t^\ve (f)-f(x_1)\bigr|>\delta \Bigr]&=& 0.
\end{eqnarray*}
Moreover, we have convergence in law on the space of cadlag paths (with
the Skorokhod topology) of $(\nu_t^\ve)_{t\ge0}$ to a jump process
$(\nu_t)_{t\ge0}$ such that:
\begin{longlist}[(ii)]
\item[(i)] (Metastability). If $A(x_1)<A(x_3)$, then $(\nu_t)_{t\ge0}$
is given by
\[
\nu_t = \cases{
\delta_{x_3}, &\quad $t < T,$
\vspace*{2pt}\cr
\delta_{x_1}, &\quad $t \geq T,$}
\]
where $T$ is an exponential mean $1$ random variable.
\item[(ii)] (Bistability). If $A(x_1)=A(x_3)$ and a sequence of
transition times is defined by   $\tT_\ve^0=0$, and
\begin{eqnarray*}
T^i_\ve&=& \inf\bigl\{t>\tT^{i-1}_\ve
\dvtx \tX_\ve(t) \in B_c(x_1)\bigr\}, \\
\tT
^i_\ve& =& \inf\bigl\{t>T^i_\ve
\dvtx \tX_\ve(t) \in B_c(x_3)\bigr\},\qquad i=1,2,
\ldots,
\end{eqnarray*}
then $(\nu_t)_{t\ge0}$ is given by
\[
\nu_t = \cases{ %
\delta_{x_3}, & \quad $T_{2i} \leq t < T_{2i+1},$
\vspace*{2pt}\cr
\delta_{x_1}, &\quad $T_{2i+1} \leq t < T_{2i+2},$}\qquad
 i=0,1,2,\ldots,
\]
where $T_0=0$, and $\{T_i\}_{i \geq0}$ are arrival times in a rate $1$
Poisson process.
\end{longlist}
\end{proposition}

\begin{pf}
(i) is simply a restatement of the main result Theorem~2 in \cite{GOV},
and (ii) is an easy extension of this result. Since $A(y_1)=A(y_3)$, we
have $\beta_\ve=\tilde{\beta}_\ve$ and the transitions from one stable
equilibrium to the other happen on the same exponential time-scale. By
Proposition~\ref{LD1}(i), on the time-scale $T=\beta_\ve t$, in the
limit as $\ve\to0$, $T_\ve^1$ is exponentially distributed with
parameter 1, and $\tX^\ve(T_\ve^1)\in D_1$. By the strong Markov
property of $\tX_\ve$, the time increment to the subsequent transition
$\tT^1_\ve-T^1_\ve$ is independent of $T^1_\ve$, and the same Theorem
implies that on the time-scale $T=\tilde{\beta}_\ve t=\beta_\ve t$,
in the
limit as $\ve\to0$, $\tT_\ve^1-T^1_\ve$ is also exponentially
distributed with parameter 1, and $\tX^\ve(\tT_1^\ve)\in D_3$. The rest
now follows from the same arguments as in the proof of Theorem~2 in
\cite{GOV}.
\end{pf}

%s3.2 #&#
\subsection{Finite-system-size effects}

The above results relied on using an additional parameter $\ve$ to
separate the scaling of the noise from the scaling of the drift,
obtaining a diffusion approximation for the limiting process first,
then applying large deviation techniques for the diffusion (\ref
{sdeLD}) with small perturbation coefficient $\ve$. A priori, there is
no reason why the limits need be taken in that order. Another approach
is to apply large deviations techniques directly to the rescaled
process $X_N=N^{-1}X_A$, and obtain results that describe the large
time-scale behavior of $X_N$ relative to the equilibrium points of the
limiting drift (\ref{doublewell}). It is natural to compare these
results to those for the associated diffusion with small diffusion
coefficient. We will identify the exact relation of time-scales of the
reaction system and the splitting mechanism for which large deviation
rates of these two methods can be compared.

%processes: for splitting with {\it bounded} jump sizes we again have
%convergence of the logarithmic moment generating function $g_N$, but
%do we also have convergence of the Legendre transform $\ell_N$ (in
%general variational form)? can use Kurtz\&Feng exponential generator
%limit even for {\it arbitrary} splitting? will have no nice formula}

This question is most easily answered when the reactions and the
splitting/resampling mechanism make only unit net changes at each step,
so that $X_A$ is a birth--death process. Assume, as before, all the
reaction constants have the standard scaling $\ka_\zeta^{ab}=\tka
_\zeta
^{ab}N^{1-(a+b)}$, and assume again the splitting rate is of the form
$\gamma(x,N,\ve)=\ve^2_N\gamma(x,N)$, except now the parameter $\ve
^2_N$ depends on $N$ as well. Since, by assumption~(6), the change due to
splitting is unbiased, we have $p_{x,x+1}=p_{x,x-1}=\frac{1}2$, and the
splitting variance is $\sigma_N^2=1$. As earlier $\gamma(x,N)$ is
assumed to satisfy condition in assumption (${7}^*$), that is, $\sup_x|\gamma(x,N)N^{-2}-\tgamma^2\tsig^2(\frac{x}N)|\to0$.

Suppose $X_A$ is a Markov jump process with rates $N\tr_+(x)\,dt=\bP
[X_A(t+dt)=x+1|X_A(t)=x]$ and $N\tr_-(x)\,dt=\bP[X_A(t+dt)=x-1|X_A(t)=x]$
such that $\ln\tr_+, \ln\tr_-$ are bounded Lipschitz continuous
functions, and $X_N=N^{-1}X_A$ is its rescaled version. Then, according
to the Freidlin--Wentzell large deviation theory for Markov jump
processes, \cite{SW} Theorem~6.17, since transitions between the two
stable equilibria $x_1,x_3$ of $\tphi$ are uniquely achieved by
crossing the potential barrier at $x_2$, the deviations of $X_N$ away
from neighborhoods of $x_1$ and $x_3$ are characterized by the large
deviation rate function for $X_N$ given by the quasipotential (with
respect to $x_i$ and $x_2$),
\begin{eqnarray}
&&\imath_{x_i,x_2}(\tr_+,\tr_-)\nonumber\\
&&\qquad:=\inf_{T>0} \inf
_\xi \biggl\{ \int_0^T
\ell \bigl(\xi(u),\xi'(u)\bigr)\,du \Big| \xi\in C^1
\bigl([0,T]\bigr), \xi(0)=x_i,\xi (T)=x_2 \biggr\},\nonumber\\
\eqntext{i=1,3,}
\end{eqnarray}
where $\ell$ is the action functional in variational form
\[
\ell(x,y)=\sup_{\theta} \bigl\{\theta y - \bigl(\tr_+(x)
\bigl(e^{\theta
}-1\bigr)+\tr _-(x) \bigl(e^{-\theta}-1\bigr) \bigr)
\bigr\}
\]
determined from the jump rates of the process $\tr_+$ and $\tr_-$.
Calculus of variations results, see \cite{SW} Theorem~11.15, give an
explicit expression for the quasipotential as
%
%e18 #&#
\begin{equation}
\label{bdquasipot} \imath_{x_i,x_2}(\tr_+,\tr_-)=\int_{x_i}^{x_2}
\ln \biggl(\frac{\tr
_-(x)}{\tr_+(x)} \biggr)\,dx,\qquad  i=1,3.
\end{equation}

If $X_A$ is a birth--death process whose rates $r_+(x),r_-(x)$ are such
that $r^N_+(x)=N^{-1}r_+(Nx)\to\tr_+(x)$ and
$r^N_-(x)=N^{-1}r_-(Nx)\to
\tr_-(x)$ uniformly in $x\in[0,1]$, then the logarithmic
moment-generating function $g_N(x,\theta)$ of the jump measure $\mu
_N(x, \cdot)=r^N_+(x)\delta_{1}+r^N_-(x)\delta_{-1}$, for fixed
$\theta
$, also converges uniformly in $x\in[0,1]$
\begin{eqnarray*}
g_N(x,\theta)&=&\int\bigl(e^{\theta z}-1\bigr){
\mu_N}(x,dz)=r^N_+(x) \bigl(e^{\theta
}-1
\bigr)+r^N_-(x) \bigl(e^{-\theta}-1\bigr)
\\
&\mathop{\longrightarrow}\limits_{N\to\infty}&\tr_+(x) \bigl(e^{\theta
}-1
\bigr)+\tr_-(x) \bigl(e^{-\theta}-1\bigr)=\int\bigl(e^{\theta z}-1\bigr){
\mu }(x,dz)=g(x,\theta)
\end{eqnarray*}
to the logarithmic moment generating function of the jump measure $\mu
(x,\cdot)=\tr_+(x)\delta_{1}+\tr_-(x)\delta_{-1}$. Since the Legendre
transform $\ell_N(x,y)$ of $g_N(x,y)$ has the explicit form
\begin{eqnarray*}
\ell_N(x,y)&=&\sup_{\theta} \bigl\{\theta y
-g_N(x,\theta) \bigr\}
\\
&=&\ln \biggl(\frac{y+\sqrt{y^2+4r^N_+(x)r^N_-(x)}}{2r^N_+(x)} \biggr)\\
&&{}-\sqrt {y^2+4r^N_+(x)r^N_-(x)}+r^N_+(x)+r^N_-(x)
\end{eqnarray*}
for fixed $y$, we also have uniform convergence in $x\in[\delta
,1-\delta
]$, for any $\delta>0$,
\[
\ell_N(x,y)=\sup_{\theta} \bigl\{\theta y
-g_N(x,\theta) \bigr\} \mathop {\longrightarrow} _{N\to\infty}
\sup_{\theta} \bigl\{\theta y -g(x,\theta) \bigr\}=\ell(x,y).
\]
Consequently, the large deviation behavior for $X_N=N^{-1}X_A$ is
determined by the same action functional $\ell(x,y)$ and exit times in
terms of the same quasipotential $\imath_{x_i,x_2}(\tr_+,\tr_-), i=1,3$
as above.

For the system of reactions and splitting, birth and death rates for
the process $X_N$, $r_+$ and $r_-$, respectively, are of the form
%
%e19 #&#
%e20 #&#
\begin{eqnarray}
\label{bdrates1} r_+(x)&=&N\sum_{(a,b,1)\in\calI}
\tka_1^{ab} \biggl(\frac{x}{N} \biggr)_{a,N}
\biggl(1-\frac{x}{N} \biggr)_{b,N} + \frac{1}2
\ve^2_N\gamma (x,N),
\\
\label{bdrates2} r_-(x)&=& N\sum_{(a,b,-1)\in\calI}
\tka_{-1}^{ab} \biggl(\frac
{x}{N} \biggr)_{a,N}
\biggl(1-\frac{x}{N} \biggr)_{b,N} + \frac{1}2
\ve^2_N\gamma(x,N).
\end{eqnarray}
We wish to obtain results for the time-scale of exit from a
neighborhood of a stable equilibrium for the rescaled process $X_N$
that are analogous to those for $\tX_\ve$ obtained in
Proposition~\ref
{LD1}. To this end, we will have to make some assumptions about the
behavior of $r_+$ and $r_-$ in order to use the quasipotential $ \imath
_{x_i,x_2}(\tr_+,\tr_-)$.
Let $\beta_{\ve_N}$ and $\tilde{\beta}_{\ve_N}$ denote time-scales
of the
transitions of the process $X_N$ from $D_3$ to $B_c(x_1)$, and from
$D_1$ to $B_c(x_3)$, respectively, in the analogous way as $\beta_\ve$ and
$\tilde{\beta}_\ve$ were for the singularly perturbed diffusion. The next
result establishes the time-scale of transition for $X_N$ from one
stable equilibrium to the other.

%pr3.3 #&#
\begin{proposition}\label{LD3}
If $X_A$ is a birth--death chain, whose rates satisfy %{xxx - instead
%of $N\ve^2_N\to c_\ve\ge0$ which we will assume in the corollary
%when we separate scales and compare with singular diffusion}
%e21 #&#
\begin{equation}
\frac{r_+(N\cdot)}{N}\to\tr_+(\cdot),\qquad \frac{r_-(N\cdot)}{N}\to \tr _-(\cdot)\qquad \mbox{uniformly in }[0,1]
\end{equation}
such that $\tphi=\tr_+-\tr_-$ satisfies (\ref{doublewell}), then the
mean times $\beta_{\ve_N}$ and $\tilde{\beta}_{\ve_N}$ for
transitions of
$X_N$ from $D_3$ to $B_c(x_1)$, and from $D_1$ to $B_c(x_3)$,
respectively, are given in terms of $ \imath_{x_i,x_2}(\tr_+,\tr_-)$
from (\ref{bdquasipot}) by
\[
\lim_{N \to\infty} \frac{1}N \ln\beta_{\ve_N} = \imath
_{x_3,x_2}(\tr_+,\tr _-),\qquad \lim_{N \to\infty} \frac{1}N
\ln\tilde{\beta}_{\ve_N} = \imath _{x_1,x_2}(\tr_+,\tr_-).
\]
\end{proposition}

\begin{pf}
This is just the statement of results for the exit problem for the jump
Markov chain $X_N$ in terms of its quasipotential, obtained by Freidlin
and Wentzell~\cite{FW}; see Theorems~1.2 and~2.1 of Chapter~5,
the discussion at the beginning of Section~4 and Theorem~4.3 of Chapter~5;
also see Theorem~5.7.11 of Chapter~5 in \cite{DZ}. Uniform
convergence of the action potential, that is, the Legendre transform
$\ell_N$, is necessary in order to express the quasipotential $\imath
_{x_i,x_2}$ in terms of the limiting rates $\tr_+,\tr_-$. All of the
assumptions on the equilibrium points of $\tphi(x)=\tr_+(x)-\tr_-(x)$
in (\ref{doublewell}) are also necessary, since $\tphi$ determines the
fluid limit of the jump Markov chain $X_N$.
\end{pf}

Finally, we can establish the time-scale separation under which the
switching results for the rescaled jump process $X_N$ and the diffusion
$\tX_\ve$ with the small diffusion coefficient can be compared.

%th3.1 #&#
\begin{theorem}\label{corLD} If the reaction system has increments of
size $\{1,-1\}$ only, its rates have standard scaling $\ka_\zeta
^{ab}=\tka_\zeta^{ab}N^{1-(a+b)}$, its limiting drift $\tphi$ satisfies
(\ref{doublewell}) and if the splitting mechanism has increments of
size $\{1,-1\}$, its rate is $\ve^2_N\gamma(x,N)$ where $\gamma(x,N)$
satisfies assumption (${7}^*$) and
\[
N\ve^2_N\to1,
\]
then results based on large deviations for $X_N$ in Proposition~\ref
{LD3} are more informative than results based on large deviations for
the diffusion $\tX_\ve$ with the small perturbation parameter $\ve_N$
in Proposition~\ref{LD1}, that is,
\[
\imath_{x_i,x_2}(\tr_+,\tr_-)\le I_{x_i,x_2}(\tphi,\tgamma\tsig).
\]
\end{theorem}

\begin{pf}
For $r^N_+(x)=N^{-1}r_+(Nx)$ and $r^N_-(x)=N^{-1}r_-(Nx)$ by (\ref
{bdrates1})--(\ref{bdrates2}), we have
\begin{eqnarray*}
r^N_+(x)&=&\sum_{(a,b,1)\in\calI}
\tka_1^{ab}(x)_{a,N}(1-x)_{b,N} +
\frac{1}2N^{-1}\ve^2_N\gamma(Nx,N),
\\
r^N_-(x)&=&\sum_{(a,b,-1)\in\calI}
\tka_{-1}^{ab}(x)_{a,N}(1-x)_{b,N} +
\frac{1}2N^{-1}\ve^2_N\gamma(Nx,N).
\end{eqnarray*}
Since $\gamma(x,N)$ is such that $|\gamma(Nx,N)N^{-2}-\tgamma^2\tsig
^2(x)|\to0$ uniformly in $x\in\{0,\frac{1}N,\ldots,1\}$, then given that
$N\ve^2_N\to1$, we have uniform convergence of $r^N_+\to\tr_+$ and
$r^N_+\to\tr_+$ to
\begin{eqnarray*}
\tr_+(x)&=&\sum_{(a,b,1)\in\calI}\tka_1^{ab}x^a(1-x)^b
+ \frac{1}2\tgamma ^2\tsig^2(x),
\\
\tr_-(x)&=&\sum_{(a,b,-1)\in\calI}\tka_{-1}^{ab}x^a(1-x)^b
+ \frac{1}2\tgamma^2\tsig^2(x).
\end{eqnarray*}
Let $ \om(x)=1-\frac{\tr_-(x)}{\tr_+(x)} $, so
\begin{eqnarray*}
\om(x)&=&\frac{\tr_+(x)-\tr_-(x)}{\tr_+(x)}=\frac{\tphi(x)}{\sum_{(a,b,1)\in\calI}
\tka_{1}^{ab}x^a(1-x)^b + ({1}/2)\tgamma^2\tsig
^2(x)},
\\
\frac{\om(x)}{1-\om(x)}&=&\frac{\tr_+(x)-\tr_-(x)}{\tr
_-(x)}=\frac{\tphi
(x)}{\sum_{(a,b,-1)\in\calI}\tka_{-1}^{ab}x^a(1-x)^b + ({1}/2)\tgamma
^2\tsig^2(x)}
\end{eqnarray*}
and (\ref{bdquasipot}) implies that $\imath_{x_i,x_2}(\tr_+,\tr
_-)=\int_{x_i}^{x_2}\ln (1-\om(x) )\,dx$ satisfies
\begin{eqnarray*}
&&-\int_{x_i}^{x_2} \frac{\tphi(x) \,dx}{\sum_{(a,b,-1)\in
\calI}\tka_{-1}^{ab}x^a(1-x)^b + ({1}/2)\tgamma^2\tsig^2(x)} \\
&&\qquad\le \imath
_{x_i,x_2}(\tr_+,\tr_-) \\
&&\qquad\le-\int_{x_i}^{x_2}
\frac{\tphi(x)
\,dx}{\sum_{(a,b,1)\in\calI}\tka_{1}^{ab}x^a(1-x)^b + ({1}/2)\tgamma^2\tsig^2(x)}.
\end{eqnarray*}
On the other hand by (\ref{potdiff2}) and the fact that $g'(y)=\tgamma
\tsig(g(y))$ we also have
\begin{eqnarray*}
I_{x_i,x_2}(\tphi,\tsig)&=&-2\int_{y_i}^{y_2}
\al(y)\,dy=-2\int_{y_i}^{y_2}\frac{\tphi(g(y)) \,dy}{\tgamma\tsig(g(y))}=-\int
_{x_i}^{x_2}\frac{\tphi(x) \,dx}{({1}/2)\tgamma^2\tsig^2(x)}\\
& \ge& \imath
_{x_i,x_2}(\tr_+,\tr_-).
\end{eqnarray*}
Hence if $N\ve_N^2\to1$, we get a comparison using quasipotentials for
$X_N$ and $\tX_\ve$ of the time-scales for transitions between stable
equilibria, as
\[
\ln\beta_\ve\approx\frac{1}{\ve_N^2}I_{x_i,x_2}(\tphi,\tsig)
\gtrsim N\imath_{x_i,x_2}(\tr_+,\tr_-) \approx\ln\beta_{\ve_N}.
\]
\upqed\end{pf}

If $\ve^2_N=N^{-1}$, transitions between stable equilibria are more
often due to finite-system-size effects than due to the effects of an
additional mechanism. This is understandable in light of the fact that
the diffusion $\tX_\ve$ is a limit of the rescaled process $X_N$ in
which the contribution of any subdiffusive noise disappears. As
remarked earlier, when $\ve^2_N=N^{-1}$, we could use this informally
prior to obtaining a diffusion limit $\tX_\ve$. If, for rates of
balanced reactions we write $\ka^{ab}_\zeta(N)=N^{1-(a+b)}\tka
^{ab}_\zeta=\ve^2_NN^{2-(a+b)}\tka^{ab}_\zeta$, then the diffusion
coefficient would become $\tila_\ve(x)=\ve^2_N (\sum_{(a,b,\zeta)\in
I^{\ba}}\tka^{ab}_{\zeta,\mu}x^a(1-x)^b+\tgamma^2\tsig^2(x) )$.
However, even this ``adjusted'' diffusion coefficient would not change
the conclusion of Theorem~\ref{corLD}, since the contribution of the
rates from biased reactions is still missing from the quasipotential of
$\tX_\ve$.

If $N^{-1}\ll\ve^2_N\ll1$, it is clear from Theorem~\ref{corLD} that
the noise of the splitting is the dominant factor in effecting the
transitions, while if $\ve^2_N\ll N^{-1}$, the noise from reactions
dominates, and both rates $\tr_+,\tr_-$ and the quasipotential
$\imath
_{x_i,x_2}$ are determined by the reaction system only.

%s3.3 #&#
\subsection{Example: bistable behavior from slow splitting}\label
{subsec:examplebist}
Here is an example of a simple reaction system that yields a limiting
system with a double-well potential:
%
%e22 #&#
%e23 #&#
%e24 #&#
%e25 #&#
\begin{eqnarray}
\label{dwreac1} A & \stackrel{\ka^{10}_{-1}} {\to} B,
\\
B & \stackrel{\ka^{01}_{1}} {\to} A, %\label{dwreac2}
\\
A+B & \stackrel{\ka^{11}_{-1}} {\to} 2B, \label{dwreac3}
\\
2A+B & \stackrel{\ka^{21}_{1}} {\to} 3A .\label{dwreac4}
\end{eqnarray}
The trimolecular reaction \eqref{dwreac4} produces a term in the drift
which is cubic in $X_A$, which is needed in order to obtain the three
desired equilibria. With standard mass-action scaling $\ka^{ab}_\zeta
=N^{1-(a+b)} \tka^{ab}_\zeta$, the limit of $F_N(X_N(t))=\bE[X_N(t)] =
\bE[X_A(t)/N] \in[0,1]$ as $N \to\infty$ is
\begin{eqnarray}
\tphi(x)=\lim_{N\to\infty}F_N(x) = -
\tka^{10}_{-1} x +\tka ^{01}_{1}(1-x) -
\tka^{11}_{-1} x(1-x) + \tka^{21}_{1}
x^2(1-x),\nonumber\\
 \eqntext{x\in[0,1].}
 \end{eqnarray}
With the special choice of
$\tka^{10}_{-1}=\tka^{01}_{1}=1, \tka^{11}_{-1} = \frac{16}{3},
\tka
^{21}_{1} = \frac{32}{3}$ we have
%
%e26 #&#
\begin{equation}\qquad
\label{muode} \tphi(x) =\tfrac{1}{3} \bigl( 3-22x + 48 x^2 -32
x^3 \bigr)=-\tfrac
{32}{3} \bigl(x-\tfrac{1}4 \bigr)
\bigl(x-\tfrac{1}2 \bigr) \bigl(x-\tfrac{3}4 \bigr)
\end{equation}
with two stable points at $x_1=\frac{1}4$ and $x_3=\frac{3}4$ and one
unstable point at $x_2=\frac{1}2$ for the system, and thus $\Phi=-\int
\tphi$ is a double-well potential. Since $\tphi$ is antisymmetric about
the line $x=\frac{1}2$ the potential can be expressed as
\[
\Phi(x) = \tfrac{1}{6} (2x-1)^4 - \tfrac{1}{12}(2x-1)^2
+C,
\]
which is symmetric about the line $x=\frac{1}2$, %(that is, $\Phi(x+
and thus $\Phi$ has equally deep wells $\Phi(\frac{1}4)=\Phi(\frac{3}4)$.\vspace*{1pt}

This system bears resemblance to the so-called Schl\"ogl model \cite
{Sch}, which consists of four reactions $A+2X \rightleftharpoons3X, B
\rightleftharpoons X$, with the resulting drift for $X$ cubic. In \cite
{VQ} the authors formulate the Kolmogorov forward equation (chemical
master equation) to analyze the stochastic model for this reaction system.

For this example we take the simplest splitting/resampling mechanism
[Example (Bern) in Section~\ref{subsec:splitmodel}] in which at each
split an error in the molecular count of $A$ from the parent to the
daughter cell is at most $1$. Its rate is $\gamma(x,N)=\gamma
(N)x/N(1-x/N)$ and its probabilities are $p_{x,x+1}=p_{x,x-1}=1/2$ for
$x\neq0,N$, and $p_{0,0}=p_{N,N}=1$. Note that here the factor $\gamma
(N)$ will depend on~$N$, but is state independent. This mechanism can
also be represented in terms of reactions as
%
%e27 #&#
\begin{equation}
\label{wfsplit} A+B\stackrel{{N^{-2}}\gamma(N)} {\longrightarrow}
2A,\qquad  A+B \stackrel{{N^{-2}}\gamma(N)} {\longrightarrow} 2B.
\end{equation}
We stress that this representation (\ref{wfsplit}) of the resampling in
terms of reactions is done merely to illustrate the mechanism in a
similar way to the reactions, and is not to be confused with an actual
set of biological reactions as in (\ref{dwreac1})--(\ref{dwreac4}). This
can be done in the particular case of Moran-type resampling, since the
rates of this mechanism depend on the product of both the count of $A$
and of $B$. This is a consequence of the fact that each resampling
event picks either one molecule of $A$ or one molecule of $B$ with
probabilities relative to their proportions in the cell, and replaces
it in the daughter cell with a random choice of either $A$ or $B$ with
equal probability.

As shown in Section~\ref{subsec:splitmodel}, if we choose the splitting
parameter to be $\gamma(N)=\frac{1}2\ve^2N^2$ for some small constant
$\ve
^2>0$, then $\gamma(x,N)$ satisfies all the conditions of assumption
(${7}^*$), and the limiting process
$\tX_\ve$ satisfies the stochastic differential equation with drift
(\ref{muode}) and diffusion coefficient $\ve^2\frac{1}2 x(1-x)$
%
%e28 #&#
\begin{eqnarray}
\label{example:bistsde} d\tX_\ve(t) &=&\tfrac{1}{3} \bigl( 3-22
\tX_\ve(t) + 48 \tX_\ve ^2(t) -32 \tX
_\ve(t)^3 \bigr)\,dt
\nonumber
\\[-8pt]
\\[-8pt]
\nonumber
&&{}+ \ve\sqrt{\tX_\ve(t)
\bigl(1-\tX_\ve(t)\bigr)}\,dB(t).
\end{eqnarray}
To find the value of the quasipotential for this problem we find the
transformation of the potential via $\al(y)=\tphi(g(y))/\tsig(g(y))$,
where $g$ is the solution to $g'(y)=\tsig(g(y))=\sqrt{g(y)(1-g(y))}$,
given explicitly by
\[
g(y) = \cos^2 \biggl(\frac{1}{2}\biggl(y-
\frac{\pi}{2}\biggr) \biggr) = \cos^2 \biggl(\frac{ y}{2} -
\frac{\pi}{4} \biggr), \qquad y\in \biggl[-\frac
{\pi}{2},\frac{\pi}{2}
\biggr].
\]
We chose the constant of integration so that $g(0)=\frac{1}2$, and
$g(-y)=1-g(y)$. The inverse of $g$ is given by
\[
h(x) = g^{-1}(x) = 2\arctan \biggl(-\sqrt{\frac{1}{x}-1}
\biggr)+\frac
{\pi
}{2},\qquad  x\in[0,1];
\]
hence, the transformed equilibrium points $y_i = h(x_i)$ are
\[
y_1 = 2 \arctan(-\sqrt{3})+\frac{\pi}{2 }=-\frac{\pi}{6 },\qquad
y_3 = 2 \arctan \biggl(-\sqrt{\frac{1}3} \biggr)+
\frac{\pi}{2 } = \frac{\pi}{6}
\]
and
\[
y_2 = 2 \arctan(-1)+\frac{\pi}{2 }= 0.
\]
Note as well that the wells of the transformed potential are of equal
depth, which follows from the fact that
$\al$ is an odd function
\[
\al(-y) = \frac{\tphi(g(-y))}{\tsig(g(-y))} = \frac{\tphi
(1-g(y))}{\tsig(1-g(y))} =\frac{-\tphi(g(y))}{\tsig(g(y))} = -
\al(y),
\]
and thus $A = -\int\al(y) \,dy$ is an even function. Since $y_1 = -y_3$,
$y_2=0$, and $A(y_1) = A(y_3)=0$, we have $A(y_2)=\int_0^{\pi/6}\al
(y)\,dy$ with a rather complicated expression
\begin{eqnarray*}
A(y_2)&=&\int_0^{\pi/6}
\frac{\tphi(\cos^2 (({y}/{2}) - ({\pi}/{4}) ))}{\tsig(\cos^2 (({y}/{2}) - ({\pi}/{4})
))}\,dy=\frac{1}{3}\int_{1/2}^{3/4}
\frac{3-22x + 48x^2 -32 x^3}{x(1-x)}\,dx\\
& \stackrel {\cdot} {=}& 0.0913.
\end{eqnarray*}

By Proposition~\ref{LD1} on a time-scale of $O(e^{-\ve^{-2}2A(y_2)})$,
the process exists a neighborhood of the stable equilibria $x_1=\frac{1}4, x_3=\frac{3}4$. Symmetry of $A$ around $y_2=0$ implies that we are in
the bistable case (ii) of Proposition~\ref{LD2}, and the occupation
measure of the process $\tX_\ve$ converges to the occupation measure of
a two-state Markov chain, which transitions between states $\{\frac{1}4,\frac{3}4\}$ with equal rates. Figure~\ref{fig:bistable} shows an
exact simulation of a sample path of the rescaled process
$X_N=N^{-1}X_A$ with choice of parameters $N=1500$, $\ve^2=0.02$; since
$\ve^2 \gg1/N$, we expect the $\ep$-perturbation of the limiting
diffusion to be driving the switching. Indeed, the process appears to
be spending most of its time in neighborhoods $B_{0.1}(x_1)\cup
B_{0.1}(x_3)$, switching between them at the approximate rate
$R=e^{-\ve
^{-2} 2A(y_2)}\stackrel{\cdot}{=} 0.0001083$.%\emph{xxx - this does not
%match the figure very well}
%{\it xxx - this does not seem to match $k=0.002$ for which $
%yyy - nice picture with M=3x10^7, N=1500, k=0.01 BistableApr16.svg

%L-temp

%f1 #&#
\begin{figure}

\includegraphics{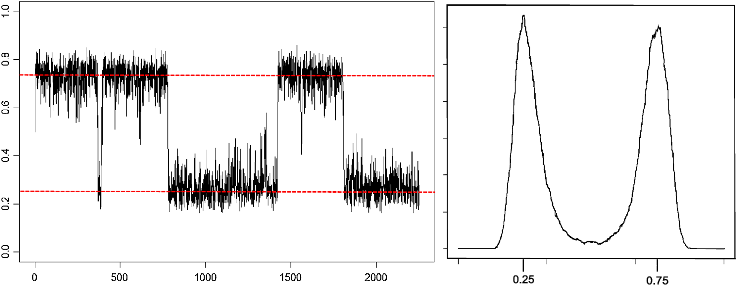}

\caption{Sample path $X_N(t)$ (left: $x$-axis${}=t$,
$y$-axis${}=X_N(t)=N^{-1}X(t)$) and its occupation density (right:
$x$-axis${}={}$state space of $X_N\subset[0,1]$, $y$-axis${}={}$proportion of time
$X_N$ spends in each state by time $t=2500$) from the system (\protect\ref
{dwreac1})--(\protect\ref{dwreac4}) with birth--death splitting, under standard
mass-action scaling for reactions and $\gamma(\ve, N)=\frac{1}2\ve^2N^2$
(parameters $N=1500$, $\ep^2=0.02$). Dashed red lines indicate
quasi-equilibria at $1/4$ and $3/4$.}
\label{fig:bistable}
\end{figure}

%f2 #&#
\begin{figure}

\includegraphics{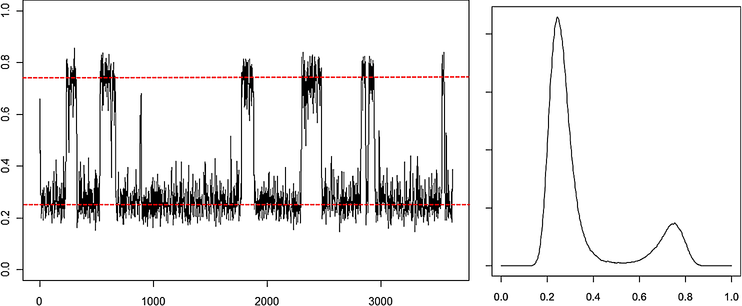}

\caption{Sample path $X_N(t)$ [left: $x$-axis${}=t$, $y$-axis${}=X_N(t)$] and
its occupation density (right: $x$-axis${}={}$state space of $X_N\subset
[0,1]$, $y$-axis${}={}$proportion of time $X_N$ spends in each state by time
$t=4000$) for the system (\protect\ref{dwreac1})--(\protect\ref
{dwreac4}) with
birth--death splitting, under standard mass-action scaling for reactions
and $\gamma(\ve, N)=\frac{1}2\ve^2N^2$ (parameters $N=500$, $\ep^2=2
\times10^{-4}$). Dashed red lines indicate quasi-equilibria at $1/4$
and $3/4$ as above.} \label{fig:bistable2}
\end{figure}

If we take $\ve^2\ll1/N$, then transitions between stable equilibria
are based only on the scaled rates for the reaction system \eqref
{dwreac1}--\eqref{dwreac4},
\[
\tr_+(x)=1-x+ \tfrac{32}{3}x^2(1-x) \quad\mbox{and}\quad \tr_-(x) = x +
\tfrac
{16}{3}x(1-x).
\]
By Proposition~\ref{LD3} the values of the quasipotential for the birth--death Markov process are
\[
\imath_{x_1,x_2}=\int_{x_1}^{x_2}\ln \biggl(
\frac{\tr_-(x)}{\tr
_+(x)} \biggr)\,dx = \int_{{1}/4}^{{1}/2} \ln
\biggl(\frac
{x+
({16}/{3})x(1-x)}{1-x+ ({32}/{3})x^2(1-x)} \biggr)\,dx \stackrel{\cdot} {=} 0.006713
\]
and
\[
\imath_{x_3,x_2}=\int_{x_3}^{x_2}\ln \biggl(
\frac{\tr_-(x)}{\tr
_+(x)} \biggr)\,dx\stackrel{\cdot} {=} 0.005534.
\]
Note that here the values for the quasipotential are no longer equal,
and the process will take longer to get out of the neighborhood of the
equilibrium $x_1=\frac{1}4$. Figure~\ref{fig:bistable2} shows a
simulation of a sample path of the rescaled process $X_N$ for $\gamma
(N)=\ep^2N$ with the choice of parameters $N=500$, but $\ve^2=2\times
10^{-4}$. In this case $1/N \gg\ep^2$ and we expect the transitions to
be due to noise from the reactions arising from finite-$N$ effects.
Based on the above calculation we expect the process to be switching
away from $B_{0.1}(x_1)$ at rate $R=e^{-N \imath_{x_1,x_2}} =
e^{-3.356629} \stackrel{\cdot}{=}0.035$ and away from $B_{0.1}(x_3)$ at
a rate $R'=e^{-N \imath_{x_3,x_2}}= e^{-2.769957}\stackrel{\cdot
}{=}0.062$; indeed, the time spent near $x_1$ is appreciably larger
than the time spent near $x_3$. %\emph{xxx - this definitely does not
%match the picture!! should really pick $k=0.0005$ so that $

We make a particular note that the reaction system considered here is
very sensitive to the exact values given for the reaction constants; a
small change in these would preserve the double-well potential, but
would lead to nonequal depth of the two wells for the quasipotential,
and hence instead of a limiting bistable behavior would lead to a
limiting metastable behavior as in case (i) of Proposition~\ref{LD2}.
In the next section we discuss the conditions on the scaling of the
reaction and splitting/resampling which yield behavior that can also be
described as bistable, but where the underlying mechanism is
qualitatively different and the restrictions on the reaction system are
negligible.

%s4 #&#
\section{Bistable behavior from fast splitting}\label{sec:switching}

We next consider the case $\ve_A \approx\infty$, and assume that time
has been rescaled so that $\tilde{c}_{\sigma^2}= \lim_{N\to\infty
}c_{\sigma^2}(N) \in(0, \infty)$ and $\tilde{c}_\mu= \lim_{N\to
\infty
}c_\mu(N) \approx0$. This is a more unconventional scaling, in which
the noise (from balanced reactions and splitting) overwhelms the
contribution due to the drift (from biased reactions).

One way to model this with a diffusion would be to introduce a
time-scale separation with an additional small parameter $\ve$ in the
scaling of all reactions rather than in the rate of splitting. Suppose
all reaction constants scale as $\ka_\zeta^{ab}=\ve\tka_\zeta
^{ab}N^{1-(a+b)}$, while the rate of splitting $\gamma(x,N)$ satisfies
assumption (${7}^*$). For any fixed $\ve>0$, the resulting limit of
the rescaled process $X_N$ would be
%
%e29 #&#
\begin{equation}
d\tX^\ve(t) = \ve\tphi\bigl(\tX^\ve(t)\bigr)\,dt+
\tgamma\tsig\bigl(\tX^\ve (t)\bigr)\,dB(t),\qquad \tX^\ve\in[0,1],
\end{equation}
where $B$ is a standard Brownian motion, and we have the case of a
diffusion with a small drift. Note that although $\tsig^2(0)=\tsig
^2(1)=0$ [by assumption~(7)], the boundaries $\{0,1\}$ are not absorbing,
since there is at least one biased reaction that allows escape from
either boundary $\tphi(0)>0, \tphi(1)<0$ [by assumption~(2)]. Other than
at the boundaries the contribution of the drift is essentially
negligible, and $\tX^\ve$ is approximately a martingale.
Most attempts to escape a boundary are followed by the return to the
same boundary point; only some end up at the opposite one. In the limit
as $\ve\to0$, the rate of escapes from the boundaries for $\tX^\ve$
vanishes, and there is no switching.\looseness=1

However, under the right conditions, the limit of the original rescaled
process will spend almost all of its time at one boundary or the other,
switching between the two on a reasonable time-scale, creating again a
bistable system. How the effect of the attempts to escape the boundary
appears in the limit depends on the rate of the attempts, and the time
spent between the boundaries.
In order to make a precise statement we need to examine the behavior of
the rescaled process $X_N$ directly and specify a general set of
conditions for a Markov jump process to exhibit this type of switching behavior.
%yyy - are we allowed to talk about \tphi yet since we do not even have
%a diffusion limit yet? Use F_N instead?

%s4.1 #&#
\subsection{Stochastic switching}\label{subsec:switch}

The unscaled process $X_A$ is a Markov chain on $\{0,1,\ldots,N\}$ with
transitions that are due to the reactions $(a,b,\zeta)\in\calI$, as
well as the splitting mechanism with distribution $p_{x,y}, {(x,y)\in\{
0,1,\ldots,N\}^2}$. The rates of these transitions from $X_A=x$ are
equal to $\sum_{(a,b,\zeta)\in\calI}\la_\zeta^{ab}(x)=\sum_{(a,b,\zeta
)\in\calI}\ka_\zeta^{ab}(N)(x)_a(N-x)_b$\vspace*{3pt} %#ff
from the reaction system and $\gamma(x,N)$ from the splitting, respectively.
We denote the total combined rate of $X_A$ from $i\in\{0,\ldots, n\}$ to
$j\in\{0, \ldots,n\}$ by
%We combine these rates by letting, for each $(i,j)\in\{0,1,\ldots,N\}
%$j$ be
\[
r_{ij}=\sum_{(a,b,j-i)\in\calI}\ka_{j-i}^{ab}(N)i_a(N-i)_b+
\gamma (i,N) p_{i,j}.
\]
Transitions due to splitting can have jumps whose size can in principle
be as large as $N-1$ (such as those of the Wright--Fisher splitting
process example in Section~\ref{subsec:splitmodel}), although with very
small probability. However, a splitting mechanism is absorbing at $\{
0,N\}$, $p_{0,0}=p_{N,N}=1$, and the rates of jumps off the boundaries
$x\in\{0,N\}$ are created by\vadjust{\goodbreak} reactions using only molecules of $B$ (for
$x=0$), or using only molecules of $A$ (for $x=N$), with rates
\[
r_{0j}= \sum_{(0,b,j)\in\calI}
\ka_j^{0j}(N) (N)_b, \qquad r_{Nj}= \sum
_{(a,0,N-j)\in\calI}\ka_{N-j}^{a0}(N)
(N)_a.
\]
By assumption~(2) in Section~\ref{subsec:reacmodel}, there exist
$j,j'\in
\{1,\ldots, N-1\}$ such that $r_{0j},r_{Nj'}\neq0$. The leading powers
of $N$, $\max_{(0,b,j)\in\calI}\{b\}>0$ and\break  $\max_{(a,0,N-j)\in
\calI}\{
a\}>0$, respectively, will determine the rate at which attempts to
counteract absorption at the boundaries happen, and in particular, this
implies that $r_{0j}$, $r_{Nj}\to\infty$ as $N\to\infty$ [allowing for
upcoming condition~(\ref{swcond1})].

Define an \textit{excursion} of $X_A$ to be any segment $X_A(t), t \in
[t_1,t_2)$ such that $X_A(t_1-),X_A(t_2) \in\{0,N\}$ and $X_A(t) \notin\{0,N\}$ for $t \in[t_1,t_2)$. Call an excursion on $[t_1,t_2)$
``successful'' if $X_A(t_1-) \neq X_A(t_2)$, and ``unsuccessful''
otherwise. For $0\leq j \leq N$, let $\tau_j:=\inf\{t\ge0\dvtx X_A(t)=j\}$
be the first hitting time of state $j$, and let $\tau_{0,N}= \tau_0
\land\tau_N$ denote the first hitting time of either boundary state. Let
\begin{eqnarray*}
e_{j0} &=&\bE\bigl[\tau_{0,N}| X_A(0)=j,
X_A(\tau_{0,N})=0\bigr],\\
e_{jN}&=&\bE \bigl[\tau
_{0,N}| X_A(0)=j, X_A(\tau_{0,N})=N
\bigr]
\end{eqnarray*}
be the expected hitting time of the two boundaries from $j$ and
$\pi_{jN}$ be the probability that an excursion from $j$ hits the $N$
boundary first
\[
\pi_{jN}=\bP\bigl[X_A( \tau_{0,N})=N|
X_A(0)=j\bigr],
\]
and thus $\pi_{j0}=1-\pi_{jN}$ is the probability it first hits the $0$
boundary.
The values of $\{e_{j\cdot},\pi_{jN}\}_{j\in\{1,\ldots,N-1\}}$ can be
determined by setting up and solving the appropriate linear functionals
of the generator for the Markov process $X_A$; explicit expressions,
however, may be hard to come by for general processes.

Excursions of $X_A$ depend on transitions from both reactions and the
splitting mechanism. However, if the noise overwhelms the drift, then
at each step in the interior transition rates are dominated by those
from the balanced reactions and the splitting mechanism. In particular,
this will imply that in the interior $X_A$ behaves approximately like a
martingale, and will allow us to approximate the probability of
switching from one boundary point to the other in terms of the relative
rates of biased reactions versus balanced reactions and splitting. We
will estimate $e_{j0},e_{jN}, \pi_{jN}$ in an example to come, and
exhibit more explicit conditions than the ones below in the case when
the reactions and splitting yield a birth--death process for $X_A$.

We first state general conditions under which the rescaled process
$X_N=X_A/N$ can be approximated by a simple Markov jump process.
Suppose that there exists two scaling parameters: the order of
magnitude of the rate of reactions on the boundary $\omega_N\to\infty
$, and a time scaling parameter $\beta_N>0$ for the rescaled process
$X_N$, such that
%
%e30 #&#
%e31 #&#
%e32 #&#
\begin{eqnarray}
\label{swcond1}&& \frac{1}{\omega_N}\sum_jr_{0j}
\to\tr_{+},\qquad \frac{1}{\omega
_N}\sum_jr_{Nj}
\to\tr_{-},
\\
&&\beta_N\sum_jr_{0j}
\pi_{jN} \to\tr_{01},\qquad \beta_N\sum
_jr_{Nj} \pi _{j0}\to
\tr_{10}, \label{swcond2}
\\
\label{swcond3}
&&\frac{1}{\beta_N\omega_N}\sum_jr_{0j}
e_{jN} ,\qquad \sum_j r_{0j}
e_{j0} ,
\nonumber
\\[-8pt]
\\[-8pt]
\nonumber
 &&\frac{1}{\beta_N\omega_N}\sum_jr_{Nj}
e_{j0},\qquad \sum_j r_{Nj}
e_{jN} \to0,
\end{eqnarray}
with $\tr_{+}, \tr_{-}, \tr_{01}, \tr_{10}\in(0,\infty)$.
Since $r_{0j}, r_{Nj}\to\infty$, there is no need to change the
time-scale for the process. These conditions imply that there are many
excursions in any finite time interval $[0,t]$, only a small fraction
of which are successful, and during which the total time spent is very
small. Consequently, the rescaled process will spend most of its time
on one boundary until the first time a successful excursion takes it to
the other boundary.
Let $\tT_N^0=\inf\{t\ge0\dvtx X_A(t)=0\}$, and
\begin{eqnarray}
T_N^i=\inf\bigl\{t>\tT_N^{i-1}
\dvtx X_A(t)=N\bigr\},\qquad \tT_N^i= \inf\bigl\{t
>T_N^i\dvtx X_A(t)=0\bigr\}, \nonumber\\
\eqntext{i=1,2\ldots}
\end{eqnarray}
be a sequence of times at which $X_A$ first reaches a boundary
different from the one where it was most recently.
Also, define the measure-valued process $(\nu^N_t)_{t\ge0}$ for some
$\rho_N>0$ such that $\frac{\rho_N}{\beta_N}\to0$ by
\[
\nu^N_t(f)=\frac{1}{\rho_N}\int_{\beta_Nt}^{\beta_Nt+\rho_N}
f\bigl(X_N(s)\bigr)\,ds
\]
for any (bounded continuous) function $f$ on $\{0,\frac{1}N,\ldots,1\}$;
this $(\nu_t^N)$ approximates the law of the location of the rescaled
process $X_N(t)=X_A(t)/N$ on a short time interval of length $\rho_N$.

%pr4.1 #&#
\begin{proposition} \label{genthm}
If $X_A$ satisfies \eqref{swcond1}--\eqref{swcond3}, then
\begin{eqnarray}
\lim_{N\to\infty}\bP \bigl[T_N^i-
\tT_N^{i-1}>t\beta_N \bigr]=e^{-\tr
_{01}t},\qquad
%  \forall t>0,\]and\[
\lim_{N\to\infty}\bP \bigl[T_N^{i+1}-
\tT_N^{i}>t\beta_N \bigr]=e^{-\tr
_{10}t}\nonumber\\
\eqntext{\forall t>0,}
\end{eqnarray}
and we have convergence in law on the space of cadlag paths (with the
Skorokhod topology) $(\nu^N_t)_{t\ge0}\Rightarrow(\nu_t)_{t\ge0}$ to
a jump process
\[
\nu_t = \cases{ %
\delta_{0}, & \quad $T_{2i} \leq t < T_{2i+1},$
\vspace*{2pt}\cr
\delta_{1}, & \quad $T_{2i+1} \leq t < T_{2i+2},$}
\qquad i=0,1,2,\ldots,
\]
where $\{T_{2i+1}-T_{2i}\}_{i\ge0}$ and $\{T_{2i+2}-T_{2i+1}\}_{i\ge
0}$ are two independent sequences of i.i.d. exponential variables with
rates $\tr_{01}$ and $\tr_{10}$, respectively.
\end{proposition}

The rescaled process $X_N$ can therefore be approximated by a jump
Markov process $(J(t))_{t\ge0}$ on $\{0,1\}$ with transition rates
$\tr
_{01}$ from ${0\to1}$, and $\tr_{10}$ from $1 \to0$ in the following
sense: the occupation times of $X_N$ on $\{0,1\}$ converge to the
respective occupation times of $J$, and the times of successful
excursions of $X_N$ from $0\to1$ and from $1\to0$ converge to the
respective transitions of $J$. We cannot expect a stronger kind of
convergence than stated, since, for example, convergence in law of
$X_N$ to $J$ in the Skorokhod topology is precluded by the fact that
for arbitrarily large $N$, there remain unsuccessful excursions of
$X_N$ that stray from their originating boundary by a distance which is
bounded away from $0$.

A different set of conditions from those in (\ref{swcond3}) for the
length of excursions away from the boundaries, where in the limit we
get four nonzero limiting constants $\tilde{e}_{01},\tilde
{e}_{00},\tilde{e}_{10},\tilde{e}_{11}$, would imply convergence to a
limiting process which spends a nontrivial fraction of time away from
the boundary. The limiting process would behave similar to a
diffusion with ``sticky'' boundaries; see \cite{KT}, Section~15.8C.

\begin{pf*}{Proof of Proposition 4.1}
For each $i\ge0$, define a sequence of times after $\tT_N^i$ at which
excursions from $0$ start $\tsig_N^{i,i'}$ and end $\ttau_N^{i,i'}$, by
letting $\ttau_N^{i,0}=\tT_N^i$, and for $i'=1,2,\ldots$
\begin{eqnarray*}
\tsig_N^{i,i'}&=&\inf\bigl\{\ttau_N^{i,i'-1}<t
\dvtx X_A(t)\neq0, X_A(t-)=0\bigr\} ,\\
\ttau_N^{i,i'}&=&\inf\bigl\{\ttau_N^{i,i'-1}<t
\dvtx X_A(t)=0, X_A(t-)\neq0\bigr\}
\end{eqnarray*}
and let $s(i)=\inf\{i'\ge1\dvtx \ttau_N^{i,i'}>T_N^{i+1}\}$ be the index
of the first excursion from $0$ that is successful, hence $\ttau
_N^{i,s(i)}=\tT^{i+1}_N$. Note that $X_A(t)=0,\forall t\in[\ttau
_N^{i,i'-1},\tsig_N^{i,i'})$ and that $ \sum_{i'\le s(i)}(\tsig
_N^{i,i'}-\ttau_N^{i,i'-1})$ is the time spent at $0$ between
successful excursions, while $X_A\neq0$ for $t\in[\tsig
_N^{i,i'},\ttau
_N^{i,i'})$, and thus $ \sum_{i'<s(i)}(\ttau_N^{i,i'}-\tsig_N^{i,i'})$
is the time spent on unsuccessful excursions.

Consider\vspace*{1pt} the time interval $[\tT_N^i,T_N^{i+1}]-\bigcup_{i'<s(i)}[\tsig
_N^{i,i'},\ttau_N^{i,i'})$ from which subintervals for unsuccessful
excursions are excised. Excursions from $0$ are started at overall rate
$\sum_{j'} r_{0j'}$, and since excursions whose first step is to $j$
are successful with probability $\pi_{jN}$, successful excursions are
started at rate $\sum_j r_{0j}\pi_{jN}$. So
\[
W_N^i:=\tsig_N^{i,s(i)}-
\tT_N^i-\sum_{i'<s(i)}\bigl(
\ttau_N^{i,i'}-\tsig _N^{i,i'}\bigr)\sim
\operatorname{exponential} \biggl(\sum_j r_{0j}
\pi_{jN} \biggr)
\]
and
\[
s(i)\sim\operatorname{geometric} \biggl(\frac{\sum_j r_{0j}\pi_{jN}}{\sum_j
r_{0j}} \biggr).
\]
Also, for any $i'<s(i)$, the unsuccessful excursion times $\ttau
_N^{i,i'}-\tsig_N^{i,i'}$ are independent and identically distributed with
\[
\bE\bigl[\ttau_N^{i,i'}-\tsig_N^{i,i'}
\bigr]=\sum_{j}\frac{r_{0j}}{\sum_{j'}r_{0j'}}\bE\bigl[
\tau_{0,N}|X_A(0)=j, X_A(\tau_{0,N})=0
\bigr],
\]
while $T_N^{i+1}-\tsig_N^{i,s(i)}$ is a subinterval for a successful
excursion with
\[
\bE\bigl[T_N^{i+1}-\tsig_N^{i,s(i)}
\bigr]=\sum_{j}\frac{r_{0j}}{\sum_{j'}r_{0j'}}\bE\bigl[
\tau_{0,N}|X_A(0)=j, X_A(\tau_{0,N})=N
\bigr].
\]

Let
\[
U_N^i:=\sum_{i'<s(i)}\bigl(
\ttau_N^{i,i'}-\tsig_N^{i,i'}\bigr)\quad \mbox
{and} \quad S_N^i:=T_N^{i+1}-
\tsig_N^{i,s(i)},
\]
so that $ T_N^{i+1}-\tT_N^i=W_N^i+U_N^i+S_N^i$. Assumption (\ref
{swcond2}) implies $W_N^i/\beta_N\Rightarrow \operatorname{exponential}(\tr_{01}$)
as $N\to\infty$.
We next show convergence for both $U_N^i\mathop{\to} 0$ and
$S_N^i\mathop{\to}0$ in probability as $N\to\infty$, which will imply
that $(T_N^{i+1}-\tT_N^i)/\beta_N \Rightarrow \operatorname{exponential}(\tr_{01}$).

We first note that
\begin{eqnarray*}
\bE\bigl[S_N^i\bigr]&=&\bE\bigl[T_N^{i+1}-
\tsig_N^{i,s(i)}\bigr]= \sum_{j}
\frac
{r_{0j}}{\sum_{j'}r_{0j'}}e_{jN}\\
&=&\frac{1}{\sum_{j'}(r_{0j'}/\omega_N)}\frac
{1}{\beta
_N\omega_N}\sum
_{j}r_{0j} e_{jN} \cdot
\beta_N;%\to0
\end{eqnarray*}
therefore, $\bE[S_N^i/\beta_N]\to0$, since the first fraction
converges to $1/\tr_{+}$, and the second to $0$, by (\ref{swcond1}) and
(\ref{swcond3}), respectively.
Similarly, for each unsuccessful excursion $1\le i'< s(i)$
\[
\bE\bigl[\ttau_N^{i,i'}-\tsig_N^{i,i'}
\bigr]= \frac{1}{\omega_N}\frac
{1}{\sum_{j'}(r_{0j'}/\omega_N)}\sum_{j}r_{0j}e_{j0},
\]
and since $s(i)$ is geometric,
\[
\bE\bigl[s(i)\bigr]=\frac{\sum_{j} r_{0j}}{\sum_{j} r_{0j}\pi_{jN}}={\omega _N}
\frac{\sum_{j} r_{0j}/\omega_N}{\sum_{j} r_{0j}\pi_{jN}}.%\to0
\]
We have
\[
\bE\bigl[U_N^i\bigr]= \bE \biggl[\sum
_{i'<s(i)}\bigl(\ttau_N^{i,i'}-\tsig
_N^{i,i'}\bigr) \biggr]\le\bE\bigl[s(i)\bigr] \bE \bigl[
\ttau_N^{i,i'}-\tsig_N^{i,i'} \bigr]=
\frac
{\sum_{j}r_{0j}e_{j0}}{\beta_N\sum_{j} r_{0j}\pi_{jN}} \cdot\beta_N%\to0
\]
and so $\bE[U_N^i/\beta_N]\to0$, since by (\ref{swcond2}) the
denominator converges to $\tr_{01}$, and by~(\ref{swcond3}) the
numerator goes to $0$.
Hence for any $\delta>0$ we have $\bP[S_N^i>\delta]\le\frac{\bE
[S_N^i]}{\delta}\to0$ and $\bP[U_N^i>\delta]\le\frac{\bE
[U_N^i]}{\delta
}\to0$.

A completely analogous proof shows that $(\tT_N^i-T_N^{i})\beta
_N\Rightarrow \operatorname{exponential}(\tr_{10}$), and the claim about the
probability measure $\nu_t$ is immediate from the fact that $\bE
[U_N^i+S_N^i]\to0$.
\end{pf*}

To verify condition (\ref{swcond1}) one only needs to use the rates of
biased reactions on the boundary. For (\ref{swcond2}), note the fact
that if not for biased reactions, the process would be a martingale; if
the rates of the biased reactions are overpowered by those of the
balanced reactions and splitting [as quantified in \eqref{swcond2}],
then the process is approximately a martingale. Conditions in (\ref
{swcond3}) predominantly depend on how fast the rates of the balanced
reactions and splitting are, as they determine the length of excursions
of the process.

These conditions are the easiest to verify when the reactions as well
as splitting/resampling mechanism make only unit net changes at each
step, so that $X_A$ is a birth--death process with $r_{ij}=0$ if
$|i-j|>1$. In this case one can specify more precise conditions on the
rates $r_{ij}$ that will ensure that (\ref{swcond1})--(\ref{swcond3})
hold. We consider the case when the time is already rescaled, that is,
$\beta_N=1$, and the rate of reactions on the boundaries is $\omega
_N=N$. We use the following notation for birth and death rates:
\[
r_+(i):= r_{i(i+1)},\qquad r_-(i):= r_{i(i-1)},\qquad
\ve_N(i)= \frac
{r_-(i)}{r_+(i)}-1,
\]
with $\ep_N(i)$ quantifying the strength of the bias at state $i$ [we
stress its dependence on $N$ via transition rates $r_\pm(i)$].

%pr4.2 #&#
\begin{proposition}\label{specthm}
If $X_A$ is a birth--death chain whose rates satisfy %{xxx - can we
%substitute $\omega_N$ for $N$?}
%e33 #&#
%e34 #&#
%e35 #&#
\begin{eqnarray}
\label{bdcond1} \frac{r_+(0)}{N}&\to&\tr_{+}\in(0,\infty),\qquad
\frac{r_-(N)}{N}\to \tr _{-}\in(0,\infty),
\\
\sum_{i=1}^{N-1}\bigl|\ve_N(i)\bigr|
&\to&0\quad\mbox{and}\label{bdcond2}
\\
 \sum_{i=1}^{N-1}
\frac{N-i}{r_+(i)} &\to&0,\qquad \sum_{i=1}^{N-1}
\frac{i}{r_-(i)} \to0,\label{bdcond3}
\end{eqnarray}
then conditions \eqref{swcond1}--\eqref{swcond3} hold with $\omega
_N=N$, $\beta_N=1$ and $\tr_{01}=\tr_+$, $\tr_{10} = \tr_-$.
\end{proposition}
Analogous to the general case, (\ref{bdcond1}) depends only on the
rates of biased reactions on the boundaries, (\ref{bdcond2}) reflects
the fact that off of the boundaries the drift of the biased reactions
is much weaker than the noise of the balanced reactions and splitting
and (\ref{bdcond3}) is an estimate on the speed of the balanced
reactions and splitting.

\begin{pf*}{Proof of Proposition \ref{specthm}} (\ref{swcond1}) is immediate from (\ref{bdcond1}) and
$\omega_N=N$. To verify (\ref{swcond2}) we solve for $\pi_{jN}, j\in
\{
1,\ldots,N-1\}$.
%
%le4.1 #&#
\begin{lemma} \label{lemma1}
If (\ref{bdcond2}) holds, then  $N \pi_{1N} \to1$ and $N\pi_{(N-1)
0}\to1$.
\end{lemma}
\begin{pf}
Let $\vp$ be such that $\vp(X_A)$ is a martingale, that is, let $\vp(x)
= \bE[\vp(X_A(\tau_{0,N}))|X(0)=x]$ for $x\in\{1,\ldots, N-1\}$ and
$\vp
(0)=0, \vp(1)=1$. Standard result for birth--death processes, using a
recursive equation for $\psi(x)=\vp(x)-\vp(x-1)$, gives
\[
\vp(x) = \sum_{i=1}^{x} \psi(i) = \sum
_{i=1}^{x} \prod
_{j=1}^{i-1} \frac
{r_-(j)}{r_+(j)}.
\]
By the optional stopping theorem for the stopping time $\tau_{0,N}$,
\[
\vp(i) = \bE\bigl[\vp\bigl(X_A(\tau_{0,N})\bigr)|X(0)=i
\bigr] = \pi_{i0}\vp(0)+\pi _{iN}\vp(N),
\]
so $\pi_{iN} = {(\vp(i)-\vp(0))}/{(\vp(N)-\vp(0))}=\vp(i)/\vp
(N)$, and
\[
\pi_{1N} = \frac{1}{\vp(N)}= \Biggl(\sum
_{i=1}^{N} \prod_{j=1}^{i-1}
\frac
{r_-(j)}{r_+(j)} \Biggr)^{-1}= \Biggl(\sum
_{i=1}^{N} \prod_{j=1}^{i-1}
\bigl(1+\ve_N(j)\bigr) \Biggr)^{-1}=\frac{1}{N c(N)},
\]
where $c(N)=\frac{1}N \sum_{i=1}^{N} \prod_{j=1}^{i-1} (1+\ve_N(j))$.

Condition (\ref{bdcond2}) implies that $\sup_{1\leq j \leq N-1} \{
|\ep
_N(j)|\} \to0$, so let $N_0$ be such that $\forall N>N_0$ and $\forall
j \in\{1,\ldots,N-1\}$, $|\ve_N(j)| < 1/3$. Since $\forall x \in
[0,1/3), 1-x \geq e^{-x-x^2}$, and $\forall x\in\mathbb{R}, 1+x \leq
e^x$, we have that uniformly for all $1 \leq a,b \leq N-1$, where $N>N_0$
%
%e36 #&#
\begin{equation}\qquad
\label{expineq1} \prod_{j=a}^b \bigl(1+
\ve_N(j)\bigr) \leq\prod_{j=1}^{N-1}
\bigl(1+\bigl|\ve_N(j)\bigr|\bigr) \leq \prod_{j=1}^{N-1}
e^{|\ve_N(j)|} = \exp \Biggl(\sum_{j=1}^{N-1}
\bigl|\ve _N(j)\bigr| \Biggr)
\end{equation}
and
%
%e37 #&#
\begin{eqnarray}\label{expineq2}
\prod_{j=a}^b \bigl(1+
\ve_N(j)\bigr) &\geq&\prod_{j=a}^{b}
\bigl(1-\bigl|\ve_N(j)\bigr|\bigr) \geq \prod_{j=1}^{N-1}
\bigl(1-\bigl|\ve_N(j)\bigr|\bigr)
\nonumber
\\
 &\geq&\prod_{j=1}^{N-1} \exp
\bigl(-\bigl|\ve _N(j)\bigr|-\bigl|\ve _N(j)\bigr|^2 \bigr) \\
&= &\exp
\Biggl(-\sum_{j=1}^{N-1} \bigl(\bigl|
\ve_N(j)\bigr|+\bigl|\ve _N(j)\bigr|^2\bigr) \Biggr)\nonumber
\end{eqnarray}
hence $N\pi_{1N}=1/c(N)\to1$.\vadjust{\goodbreak}

To get $N\pi_{(N-1) 0}\to1$, if we flip the state space by letting
$\check{\iota}=N-i$, then the new boundaries are $\check{0}=N$ and
$\check{N}=0$, and we get a birth--death process $\check{X}_A$ whose
rates are precisely the flip of those for $X_A$. That is, the rates of
$\check{X}_A$ are $\check{r}_+(\check{\iota})=r_-(N-i)$,  $\check
{r}_-(\check{\iota})=r_+(N-i)$, and their ratio is
\[
1+\check{\ve}_N({\check{\iota}})=\frac{\check{r}_-(\check{\iota
})}{\check{r}_+(\check{\iota})}=
\frac{r_+(N-i)}{r_-(N-i)},
\]
giving the same product of ratios as for the original process.
\[
\prod_{\check{\iota}=1}^{N-1}\bigl(1+\check{
\ve}_N({\check{\iota }})\bigr)=\prod_{j=1}^{N-1}
\bigl(1+\ve_N(j)\bigr).
\]
Hence, the exact argument above now applied to $\check{X}_A$ gives
$N\check{\pi}_{\check{1}\check{N}}=\break N\pi_{(N-1)0}\to1$ as well.
\end{pf}

% e^{\sum_{j=1}^{N-1} \frac{\ve_N(j)}{1+\ve_N(j)}}=\prod_{j=1}^{N-1}

Once we have the result of Lemma~\ref{lemma1}, it is immediate that
$N\pi_{1N}\to1, N\pi_{(N-1)0}$ imply (\ref{swcond2}) with $\tr
_{01}=\tr
_{+}, \tr_{10}=\tr_{-}$.

To verify (\ref{swcond3}) we next solve for $e_{j}, j\in\{1,\ldots
,N-1\}
$, where $e_j=\bE[\tau_{0,N}|\break  X_A(0)=j]$ for $j\in\{1,\ldots,N-1\}$, and
$e_0=e_N=0$.
%
%le4.2 #&#
\begin{lemma}\label{lemma2}
If (\ref{bdcond2}) and (\ref{bdcond3}) hold, then $Ne_1 \to0$ and $N
e_{N-1} \to0$.
\end{lemma}

\begin{pf}
The expected time of an excursion satisfies the recursion
\[
e_i = \frac{1}{r_-(i)+r_+(i)} + \frac{r_-(i)}{r_-(i)+r_+(i)} e_{i-1} +
\frac{r_+(i)}{r_-(i)+r_+(i)}e_{i+1},
\]
which gives
\[
r_+(i) (e_{i+1}-e_i) - r_-(i) (e_i-e_{i-1})
= -1;
\]
letting $f(i) = e_i-e_{i-1}$ gives the recursive equation
\[
f(i+1) = -\frac{1}{r_+(i)} + \frac{r_-(i)}{r_+(i)}f(i)= -
\frac
{1}{r_+(i)} + \bigl(1+\ve_N(i)\bigr)f(i).
\]
Note that except for the $ -\frac{1}{r_+(i)}$ term, this is reminiscent
of the recursion for $\psi(i) = \frac{r_-(i)}{r_+(i)} \psi(i-1)$. Hence
\[
f(k) = f(1)\prod_{j=1}^{k-1}
\bigl(1+\ve_N(j)\bigr) -\sum_{i=1}^{k-1}
\frac
{1}{r_+(i)} \prod_{j=i+1}^{k-1}
\bigl(1+\ve_N(j)\bigr).
\]
To find $f(1)=e_1-e_0=e_1$ we impose the condition $\sum_{i=1}^N f(i) =
e_N-e_0 = 0$ and get
\[
e_1 = \Biggl(\sum_{k=1}^N
\sum_{i=1}^{k-1} \frac{1}{r_+(i)} \prod
_{j=i+1}^{k-1} \bigl(1+\ve_N(j)
\bigr) \Biggr) \bigg/ \Biggl( \sum_{k=1}^N\prod
_{j=1}^{k-1} \bigl(1+\ve_N(j)
\bigr) \Biggr).
\]

Let\vspace*{1pt} $\eta_N=\sup_{1\leq a,b \leq N-1 } |\prod_a^b(1+\ep_N(j)) -
1
|$. Then (\ref{bdcond2}) implies $\eta_N \to0$ for $N>N_0$ via
\eqref
{expineq1} and \eqref{expineq2}. We have
\begin{eqnarray*}
e_1 &\leq& \Biggl(\sum_{k=1}^N
\sum_{i=1}^{k-1} \frac{1}{r_+(i)} (1+\eta
_N) \Biggr) \bigg/ \Biggl( \sum_{k=1}^N
(1-\eta_N) \Biggr)
\\
&=& \frac{1+\eta_N}{(1-\eta_N)N}\sum_{k=1}^N \sum
_{i=1}^{k-1} \frac
{1}{r_+(i)} =
\frac{1+\eta_N}{(1-\eta_N)N}\sum_{i=1}^{N-1}
\frac{N-i}{r_+(i)},
\end{eqnarray*}
and thus (\ref{bdcond3}) implies $Ne_1 \to0$.

To obtain $Ne_{N-1} \to0$ we can flip the process and consider $\check
{X}_A$ with the flipped rates as in the proof of Lemma~\ref{lemma1}.
Now, in addition to (\ref{bdcond2}), we also require the flip version
of the first condition in (\ref{bdcond3}),
\[
\sum_{\check{\iota}=1}^{N-1} \frac{\check{N}-\check{\iota
}}{\check
{r}_+(\check{\iota})}=\sum
_{i=1}^{N-1} \frac{N-i}{r_-(N-i)}=\sum
_{i=1}^{N-1} \frac{i}{r_-(i)} \to0,
\]
which are guaranteed by the second condition in (\ref{bdcond3}).
\end{pf}

Once we have the results of both Lemmas \ref{lemma1} and~\ref
{lemma2}, we can deduce that $Ne_1\to0$ and $Ne_{N-1}\to0$, which
imply (\ref{swcond3}). Namely, from $e_1=e_{1N}\pi_{1N}+e_{10}\pi_{10}$,
\[
e_{1N}\le\frac{e_1}{\pi_{1N}}=\frac{Ne_1}{N\pi_{1N}}\to0
\]
since Lemma~\ref{lemma1} ensures convergence of the denominator to $1$
and Lemma~\ref{lemma2} of the numerator to 0. Similarly
\[
Ne_{10}\le\frac{Ne_1}{\pi_{10}}\le\frac{Ne_1}{r_-(0)/(r_+(0)+r_-(0))} \to0
\]
since $\pi_{10}$ contains the positive probability (independent of $N$)
of an immediate return to $0$.
\end{pf*}

For a reaction system and splitting with unit net changes only, since
splitting is unbiased we have $p_{i,i+1}=p_{i,i-1}=\frac{1}2$, for $i\ne
0,N$, and the contribution to $r_+(i)$ and $r_-(i)$ from splitting is
$\frac{1}2\gamma(i,N)$. Let us write $\gamma(i,N)=\gamma(N)p_i$ where
$\gamma(N)$ depends on $N$ only (i.e., is state independent) and
$p_i=O(1)$. Then, in any state, the contribution of the splitting is of
$O(\gamma(N))$, while the contribution of the reaction system is of
$O(N)$ due to the standard scaling of reaction rates. Hence, we have
the following result.

%th4.1 #&#
\begin{theorem}\label{bdcor}
If the reaction system has increments of size $\{1,-1\}$ only; contains
reactions $aA \to(a-1)A+B$, $bB \to A+(b-1)B$ some $a,b>0$; has rates
with standard scaling
$\ka^{ab}_\zeta(N)=\tka_\zeta^{ab}N^{1-(a+b)}$;\vadjust{\goodbreak}
and if the splitting mechanism has increments of size $\{1,-1\}$,
$p_{0,0}=p_{N,N}=1$, with rate is $\gamma(i,N)=\gamma(N)p_i$ where
$\gamma(N)$ and $p_i=O(1)$ satisfy
%
%e38 #&#
\begin{equation}\qquad
\label{corcond} \frac{N}{\gamma(N)}\sum_{i=1}^{N-1}
\frac{1}{p_i}\to0, \qquad\frac
{1}{\gamma
(N)}\sum_{i=1}^{N-1}
\frac{i}{p_i}\to0,\qquad \frac{1}{\gamma(N)}\sum_{i=1}^{N-1}
\frac{N-i}{p_i}\to0;
\end{equation}
then the results of Proposition~\ref{specthm} apply with $\beta_N=1$
and $\tr_{01}=\sum_{(0,b,1)\in\calI}\tka^{0b}_{1},\break   \tr_{10}=\sum_{(a,0,-1)\in\calI}\tka^{a0}_{-1}$.
\end{theorem}

\begin{pf}
The transition rates for $X_A$ are given by
\begin{eqnarray}
r_+(i) &=&\frac{1}2\gamma(N) p_{i}+ N \sum
_{(a,b,1)\in\calI} \tka ^{ab}_{1}
(i/N)_{a,N}(1-i/N)_{b,N},\nonumber\\
\eqntext{ i=0,\ldots, N-1,}
\\
%#ff
r_-(i) &=& \frac{1}2\gamma(N) p_{i}+ N \sum
_{(a,b,-1)\in\calI} \tka ^{ab}_{-1}
(i/N)_{a,N}(1-i/N)_{b,N},\qquad i=1,\ldots, N.\nonumber %#ff
\end{eqnarray}
On the boundary the rates are
\[
r_+(0) = N \sum_{(0,b,1)\in\calI} \tka_{1}^{0b},\qquad
r_-(N) = N \sum_{(a,0,-1)\in\calI} \tka_{-1}^{a0},
\]
and (\ref{bdcond1}) holds with $\tr_+ = \sum_{(0,b,1)\in\calI}
\tka
_{1}^{0b},  \tr_- = \sum_{(a,0,-1)\in\calI} \tka_{-1}^{a0}$.
Also,
\begin{eqnarray*}
\ve_N(i) &=& \frac{({1}/2)\gamma(N)p_{i}+ N \sum_{(a,b,-1)\in\calI
} \tka
^{ab}_{-1} (i/N)_{a,N}(1-i/N)_{b,N}}{({1}/2)\gamma(N)p_{i}+ N \sum_{(a,b,1)\in\calI} \tka^{ab}_{1} (i/N)_{a,N}(1-i/N)_{b,N}}- 1
\\
&=& \biggl(N  \biggl(\sum_{(a,b,1)\in\calI} \tka^{ab}_{-1}
(i/N)_{a,N}(1-i/N)_{b,N} \\
&&\hspace*{20pt}{}- \sum_{(a,b,1)\in\calI} \tka^{ab}_{1}
(i/N)_{a,N}(1-i/N)_{b,N}  \biggr)\biggr)\\
&&{}\Big/
\biggl(({1}/2)\gamma(N) p_{i}+ N \sum_{(a,b,1)\in\calI} \tka^{ab}_{1}
(i/N)_{a,N}(1-i/N)_{b,N}\biggr)
\\
&\leq&\frac{ 2NR \max_{(a,b,\zeta)\in\calI}\tka_\zeta
^{ab}}{\gamma(N) p_{i}}
\end{eqnarray*}
since $\tka^{ab}_\zeta\ge0$, where $R<\infty$ is the number of
reactions in the system. Therefore
\[
\sum_{i=1}^{N-1} \bigl|\ve_N(i)\bigr|
\leq2R\max_{(a,b,\zeta)\in\calI}\bigl\{ \tka _\zeta^{ab}
\bigr\}\frac{N}{\gamma(N)}\sum_{i=1}^{N-1}
\frac{1}{ p_{i}},
\]
and the first condition in (\ref{corcond}) ensures that $\sum_{i=1}^{N-1} |\ve_N(i)|\to0$ and (\ref{bdcond2}) holds. %\emph{xxx -
%we may need a different criterion here depending on how we change (
On the other hand,
\begin{eqnarray*}
\sum_{i=1}^{N-1} \frac{i}{r_-(i)} &=& \sum
_{i=1}^{N-1} \frac{i}{({1}/2)\gamma(N)p_{i}+ \sum_{(a,b,-1)\in\calI} \tka^{ab}_{-1}
(i/N)^{a-1}(1-i/N)^b}\\
&\leq&
\frac{2}{\gamma(N)}\sum_{i=1}^{N-1}
\frac
{i}{p_{i}}
\end{eqnarray*}
and
\[
\sum_{i=1}^{N-1} \frac{N-i}{r_+(i)}  \leq
\frac{2}{\gamma(N)}\sum_{i=1}^{N-1}
\frac{N-i}{p_{i}}
\]
so the last two conditions in (\ref{corcond}) ensure that \eqref
{bdcond3} is satisfied as well.
\end{pf}

%s4.2 #&#
\subsection{Example: Bistable behavior from fast splitting}\label
{subsec:exampleswitch}
We revisit the same example of the reaction system we analyzed in
Section~\ref{subsec:examplebist}:
%
%e39 #&#
%e40 #&#
%e41 #&#
%e42 #&#
\begin{eqnarray}
\label{dwreac2} A & \stackrel{\ka^{10}_{-1}} {\to}& B,
\\
B & \stackrel{\ka^{01}_{1}} {\to}& A,
\\
A+B & \stackrel{\ka^{11}_{-1}} {\to} &2B,
\\
2A+B & \stackrel{\ka^{21}_{1}} {\to}& 3A \label{dwreac2last}
\end{eqnarray}
with the standard mass-action scaling, $\ka^{ab}_\zeta=N^{-(a+b)+1}
\tka
^{ab}_\zeta$. In this system the only reactions which counteract the
absorption on the boundaries are the first two unimolecular reactions.
Also, note that all system reactions change the molecular count of $A$
only by increments of size $1$.

We chose the same simple splitting mechanism as before, since
conditions~(\ref{bdcond2}), (\ref{bdcond3}) and (\ref{corcond}) are much
easier to verify than conditions (\ref{swcond2}), (\ref{swcond3}).
Recall that, if we were to assume $\gamma(N)=\frac{1}2\ve^2 N^2$ for some
small $\ve^2>0$, then the limiting process for $X_N$ would be the
diffusion process $\tX_\ve$ in (\ref{example:bistsde}); the splitting
noise is even less present if we were to assume $\gamma(N)=\frac{1}2N$,
as shown in Section~\ref{subsec:examplebist}.
In contrast, if we assume the rate $\gamma(N)$ grows fast enough so
that $N^2\ln{N}/\gamma(N)\to0$, then we can show that the conditions
in Proposition~\ref{specthm} are satisfied, and the behavior of the
limiting process for $X_N$ is described by a different two-state jump
Markov process.

There are only two reactions in (\ref{dwreac2})--(\ref{dwreac2last})
active on the boundaries, so $\tr_{01}=\tka^{01}_{1}$ and $\tr
_{10}=\tka
^{10}_{-1}$. To verify (\ref{corcond}), note that we have $\gamma
(i,N)=\gamma(N)p_i$ with $p_i=\frac{i}{N}(1-\frac
{i}{N})={i(N-i)}/{N^2}$, so
\[
\frac{N}{\gamma(N)}\sum_{i=1}^{N-1}
\frac{1}{p_{i}}=\frac
{N^3}{\gamma
(N)}\sum_{i=1}^{N-1}
\frac{1}{i(N-i)} = \frac{N^2}{\gamma(N)}\sum_{i=1}^{N-1}
\biggl( \frac{1}{i}+\frac{1}{N-i} \biggr)=\frac
{2N^2h_N}{\gamma(N)}
\]
using partial fractions $\frac{1}{i(N-i)} = \frac{1}{N}(\frac{1}{i} +
\frac{1}{(N-i)})$, where $h_N$ is the $N$th harmonic sum. Also
\[
\frac{1}{\gamma(N)}\sum_{i=1}^{N-1}
\frac{i}{p_{i}}= \frac
{N^2}{\gamma
(N)}\sum_{i=1}^{N-1}
\frac{1}{N-i} = \frac{N^2h_N}{\gamma(N)}
\]
and
\[
\frac{1}{\gamma(N)}\sum_{i=1}^{N-1}
\frac{N-i}{p_{i}} =\frac
{1}{\gamma
(N)}\sum_{i=1}^{N-1}
\frac{1}{i} = \frac{N^2h_N}{\gamma(N)}
\]
as well. Hence, $N^2{\ln N}/{\gamma(N)} \to0$ ensures that all
conditions in (\ref{corcond}) hold.

This example shows that for any reaction system with unit increments
whose drift has a double well potential, and for this particular choice
of the splitting mechanism, we can identify orders of magnitude for
$\gamma(N)$ that lead to different limiting behaviors:

$\bullet$  If $\gamma(N)\ll N$, bistability is caused by large
deviations of the Markov jump process, and the rescaled process
transitions between neighborhoods of the drift equilibirum points on a
time-scale of order $e^{N(\gamma(N))^{-1}\imath_{x_i,x_2}}$, with
\mbox{$N(\gamma(N))^{-1}\to\infty$}.

$\bullet$  If $\gamma(N)\sim\ve^2 N^2$, $\ve^2>0$ a constant,
bistability is caused by large deviations of a diffusion with a small
perturbation coefficient, with transitions between neighborhoods of the
drift equilibirum points on a time-scale of order $e^{\ve
^{-2}I_{x_i,x_2}}$.

$\bullet$  If $\gamma(N)\gg N^{2}\ln{N}$, bistability is caused by
excessive noise, and switching between the boundary points occurs on a
time-scale of order $1$.

%L-temp
%
%f3 #&#
\begin{figure}

\includegraphics{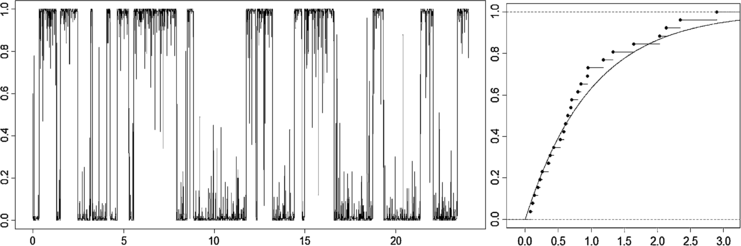}

\caption{Sample path $X_N(t)$ (left: $x$-axis${}=t$,
$y$-axis${}=X_N(t)=N^{-1}X(t)\subset[0,1]$) of the system (\protect\ref
{dwreac1})--(\protect\ref{dwreac4}) with (\protect\ref{wfsplit})
splitting, under
standard mass-action scaling for reactions and $\gamma(N)=\frac{1}2
N^{3}$ (parameter $N=200$); and the distribution of switching times
plotted (dots) in terms of quantiles (right: $x$-axis${}=t$,
$y$-axis${}={}$fraction of switching times of length $\le t$). Solid line
($1-e^{-t}$) indicates the quantiles of the exponential (mean $1$)
distribution for comparison.}
\label{fig:switching}
\end{figure}

Note that the order of magnitude $N^2$ only represents the scale on
which we have assumed that the variance of the splitting mechanism is
in the diffusive case [see assumption (${7}^*$) in Section~\ref{subsec:splitmodel}].
Also note that existence of two stable states in the deterministic
model for the reaction system is not needed for the result of this
section. We chose the same reaction system in order to make the
comparison with the results in Section~\ref{sec:bistable} and emphasize
the difference between the effects of ``slow'' and ``fast'' splitting
on the same reaction system.

Figure~\ref{fig:switching} shows an exact simulation of a sample path
of the rescaled process $X_N=N^{-1}X_A$ for a relatively short period
of time, spending most of its time at boundaries\vadjust{\goodbreak} $\{0\}\cup\{1\}$,
switching between them at approximately rates $\tr_{01}=\tka
^{01}_{1}=1, \tr_{10}=\tka^{10}_{-1}=1$; see Section~\ref{subsec:examplebist} for coefficient values. Switching between states
occurs at a time-scale $\beta_N=1$, and since $\tr_{01}=\tka
^{01}_{1}=1, \tr_{10}=\tka^{10}_{-1}=1$ the distribution of switching
times should approximately be an exponential distribution (mean $1$)
distribution. This is shown in the quantile plot in Figure~\ref{fig:switching}, where the fraction of switching times of length $\le
t$ is plotted against the same fraction $1-e^{-t}$ for the exponential
(mean $1$) distribution.

% {\it xxx- y-axis should be labeled so state $\in(0,1)$, the two
%figures should be separated, without any text or list of switch times}

%s5 #&#
\section{Discussion}

%view of experimental observations, discuss how the speed of splitting
%determines which regime we are in}

We showed that there are two different types of stochastic bistable
behavior in which the system spends most of its time at or near one of
two states and switches between them. For one of these types of
bistability, because the magnitude of noise is high, it can occur even
in a system whose deterministic model would not allow for a possibility
of bistability at all. The detreministic system can have unique stable
points, as, for example, in the neutral Wright--Fisher model with
mutation. For the other type of bistability, where the noise is
relatively low, one needs the reaction system to have two deterministic
stable points, as, for example, in the Schl\"ogl model. The important
point is what constitutes ``high'' and ``low'' levels of noise: the
determining quantity $\ve_A(N)$ (\ref{epsA}) depends on the relative
size in terms of $N$ of the variance to the average change in the
system, where $N$ is a scaling parameter for the size of the system. We
referred to $\ve_A(N)\approx0$ as ``slow'' splitting, and to $\ve
_A(N)\approx\infty$ as ``fast'' splitting, interpreted relative to the
reaction dynamics.

We discussed the differences in the qualitative signatures of
bistability in the two cases:
\begin{itemize}
\item In case of ``slow'' splitting, the states where the process
spends most of its time are determined by\vadjust{\goodbreak} the drift of the
deterministic model for the reaction system; in contrast, in case of
``fast'' splitting, they are simply the two extremes for the size of
the system.

\item In case of ``slow'' splitting, the rate of switching is
determined by the relative magnitude of the splitting variance to the
reaction drift and by the size of the potential barrier in the
deterministic model for the system; on the other hand, in case of
``fast'' splitting, the rates of switching are determined only by the
standardized rates of the reactions that are realizable from one of the
extremes for the system size.

\item In case of ``slow'' splitting, the time-scale $\beta_\ve$ or
$\beta_{\ve_N}$ on which the switching happens is exponential in (some
increasing function of) the size of the system; in contrast, in case of
``fast'' splitting, the time-scale $\beta_N$ is at most polynomial.
\end{itemize}

We also showed that the observables of bistability (switching states
and rates) are not sensitive to precise specification of the reaction
system, as they depend only on: equilibrium points, size of potential
barrier in ``slow'' splitting, and drift values at boundaries in
``fast'' splitting. However, bistability is very sensitive to the
distributional form of the splitting/resampling mechanism: the variance
of its distribution determines the potential barrier in ``slow''
splitting, and the harmonic sum of its transition probabilities
determines the threshold for appearance of ``fast'' splitting.

In the context of cellular systems of biochemical reactions, the
problem of determining the partitioning errors due to cell division is
experimentally extremely challenging (Huh and Paulsson \cite{Pa1,Pa2}). The measurements for single cells rely on count estimates for
related species rather than the molecular species of interest. In
addition, in order to estimate the magnitude of intracellular noise,
one has to separate the intrinsic from the extrinsic sources of
randomness. How random is cell division, and how it compares in
magnitude to the biochemical noise is a question that is very much
open. However, since our analysis only depends on a few general
features of the splitting mechanism (unbiasedness and time-scale of the
rate), it is also possible that stochastic bistability is achieved by a
set of auxiliary reactions, instead of splitting, acting on a different
time-scale from the rest of the system. For example, the protein
bursting mechanism may act as the driver of stochastic bistability
(Zong et al. \cite{ZSSSG}, Kaufman et al. \cite{KYMvO}).

One can try to rely on the qualitative signatures of bistability in
order to assess which of the two types of bistability we discussed is
relevant in a specific cellular biochemical systems. When the switching
times are orders of magnitude greater than the molecular count of the
switching species, as in the lysogenic switch of \textit{E. coli}, the ``slow''
splitting may be the more likely mechanism. This evaluation is
sensitive to the choice of time units, given which both the splitting
and reaction rates should be of reasonable orders of magnitude in terms
of the molecular count. It is natural to chose units of time\vadjust{\goodbreak}
corresponding to cell-doubling or cell-division time (the splitting
rate is then of order 1---and the range of splitting rates in our
model, in any of the different cases, is at most linear). In an
experimental analysis of this system, Zong et al. \cite{ZSSSG}
observed that the switching times of the cell are exponential in the
number of protein burst events, and correspond to a calculation of the
rare event probability of the bursts, as can be interpreted by large
deviations in our ``low'' auxiliary noise (``slow'' splitting) type of
bistability.
In contrast, when the switching times are relatively short, as in the
gene expression switch in \textit{S. cerevisiae}, the ``fast'' splitting is the
probable mechanism. In the engineered chemical reaction network version
of this system, Kaufmann et al. \cite{KYMvO} show that increasing the
protein burst size (increasing the auxiliary noise) leads to more
highly correlated switching behavior in different cell lineages, as
could be inferred from properties of our ``high'' auxiliary noise
(``fast'' splitting) bistability type.

%It would be natural to also study molecule counts for the entire
%population under these mechanisms, and not just for a single lineage:
%one could, for example, investigate the correlation of molecule counts
%between two cells relative to their degree of kinship. Our work here
%suggests that if the splitting is fast, then these correlations are
%high across many generations, whereas if the splitting is slow, then
%these correlations become negligible over just a few generations. {\bf
%xxx check!}

\section*{Acknowledgment}
The authors would like to thank Jonathan
Mattingly, whose suggestion during the BIRS workshop on ``Multi-scale
Stochastic Modeling of Cell Dynamics'' began this investigation.

% imsref loaded by akundreckaite, 2013-11-20 15:27:22
%


%suskaldyti doi

\printaddresses

\end{document}